\documentclass{amsart}
\usepackage{fullpage}
\usepackage[utf8]{inputenc}
\usepackage[T1]{fontenc}
\usepackage{color}
\usepackage{amsmath, amssymb, amsthm}
\usepackage{graphicx}
\usepackage{hyperref}
\usepackage{fancyhdr}
\usepackage{physics}
\usepackage{cite}
\usepackage{xcolor}
\usepackage{subfiles}
\usepackage{comment}
\usepackage{nicefrac}

\title[The effect of noise perturbation for the stochastic Burgers equation]{Quantifying the effect of noise perturbation for the stochastic Burgers equation with additive trace-class noise}
\author{Sonja Cox$^{a}$ and Matas Urbonas$^{b}$}
\thanks{$^a$: s.g.cox@uva.nl, University of Amsterdam, KdVI, Netherlands, ORCID: 0000-0002-9417-1542}
\thanks{$^b$: matas.urbonas@univ-pau.fr, Universite de Pau et des Pays de l'Adour, E2S UPPA, CNRS, LMAP, Pau, France, ORCID: 0009-0009-5265-0359}
\thanks{Funding : SC is partially supported by NWO grant VI.Vidi.213.070}
\thanks{MSC primary: 60H15, secondary: 35R60, 35A35}
\date{\today}

\newtheorem{theorem}{Theorem}[section]  
\newtheorem{lemma}[theorem]{Lemma}      
\newtheorem{proposition}[theorem]{Proposition}
\newtheorem{corollary}[theorem]{Corollary}

\theoremstyle{definition}

\newtheorem{remark}[theorem]{Remark}

\usepackage{etoolbox}

\makeatletter
\@addtoreset{theorem}{chapter}%
\@addtoreset{theorem}{section}%
\makeatother

\let\OriginalTheTheorem\thetheorem
\renewcommand{\thetheorem}{%
\ifnumequal{\number\value{section}}{0}{\thechapter.\arabic{theorem}}{%
\OriginalTheTheorem}}

\newcommand{\R}{\mathbb{R}} 
\newcommand{\N}{\mathbb{N}} 
\newcommand{\Prob}{\mathbb{P}} 
\newcommand{\calF}{\mathcal{F}}
\newcommand{\E}{\mathbb{E}}
\newcommand{\EE}[1]{\mathbb{E}\left[ #1 \right]}
\newcommand{\Ltwo}[1]{\left\| #1 \right\|_{L^{2}}}
\newcommand{\bigLtwo}[1]{\big\| #1 \big\|_{L^{2}}}

\newcommand{\calL}{\mathcal{L}}
\newcommand{\inn}[1]{\left\langle #1 \right\rangle}
\newcommand{\innt}[1]{\left\langle #1 \right\rangle_{L^2}}

\begin{document}

\begin{abstract}We establish upper bounds for the weak and strong error resulting from a perturbation of the noise driving the stochastic Burgers equation, where we assume the noise to be additive and of trace class and the initial value to be sufficiently regular. More specifically, replacing the covariance operator of the driving noise $Q_1 \in \mathcal{L}_1(L^2)$ in the Burgers equation by a covariance operator $Q_2 \in \mathcal{L}_1(L^2)$ results in a weak error of $\mathcal{O}\big(\| (-A)^{-1^{-} } (Q_1-Q_2) \|_{\mathcal{L}_1(L^2)}\big)$ and a strong error of $\mathcal{O}\big(\big\| (-A)^{-\nicefrac{1}{2}^{-}}\big|Q_1^{\nicefrac{1}{2}} -Q_2^{\nicefrac{1}{2}}\big| \big\|_{\mathcal{L}_2(L^2)}\big)$. Here $\|\cdot \|_{\mathcal{L}_1}$ is the trace class norm, $\|\cdot \|_{\mathcal{L}_2}$ is the Hilbert--Schmidt norm, and $A$ is the one-dimensional Dirichlet Laplacian that represents the leading term in the Burgers equation. In particular, our results provide upper bounds for the weak and strong error arising when approximating the trace class noise by finite-dimensional noise; the rates we obtain reflect the general philosophy that the weak convergence rate should be twice the strong rate.
\end{abstract}
\maketitle

\setcounter{tocdepth}{1}

\section{Introduction}
We consider noise perturbations for the stochastic viscous Burgers equation with homogeneous Dirichlet boundary conditions and additive trace class noise. More specifically, for $Q_1\in \mathcal{L}_1(L^2)$ a positive operator we let $\bigl(W^{Q_1}(t)\bigr)_{t\ge 0}$ be the $L^2$-valued $Q_1$-Wiener process given by
\begin{equation}  \label{eq:defWQ_intro}
W^{Q_1}(t) = \sum_{k\in\N}W^{(k)}(t) Q_1^{\nicefrac{1}{2}} h_k,\quad t\in\R^+,
\end{equation}
where $(h_k)_{k\in\N}$ is an orthonormal basis for $L^2$ and $(W^{(k)})_{k\in\N}$ is a 
sequence of independent standard real-valued Wiener processes independent of the initial value $X_0$. The Burgers equation we consider is formally given by
\begin{equation}\label{eq:Burgers_intro}
\left\lbrace
\begin{aligned}
&\partial_t X_1(t,z)=\Delta X_1(t,z)+2X_1(t,z)\nabla(X_1(t,z))+\dot{W}^{Q_{1}}(t,z),\quad t>0, z\in(0,1),\\
&X_1(t,0)=X_1(t,1)=0,\quad t>0,\\
&X_1(0,z)=X_0(z),\quad z\in(0,1),
\end{aligned}
\right.
\end{equation}
where $\nabla$ and $\Delta$ denote the first and second order derivatives with respect to $z$, and where $X_0$ is assumed to be a given (possibly random) initial value. We refer to Section~\ref{ssec:setting} for a rigorous interpretation of a solution to~\eqref{eq:Burgers_intro}. 

The goal of this article is to establish upper bounds for the strong and weak error resulting from replacing $Q_1$ by a different positive operator $Q_2\in \mathcal{L}_1(L^2)$. More specifically, let $X_{2}\colon [0,T]\times \Omega \rightarrow L^2$ be the solution to~\eqref{eq:Burgers_intro} with $W^{Q_1}$ replaced by $W^{Q_2}$; we obtain the following bounds (see Corollaries~\ref{cor:weak conv X X^{Q_2} X^{Q_2}_M} and~\ref{cor:strong conv X X^{Q_2} X^{Q_2}_M} below):

\begin{theorem}\label{thm:strong perturbation intro}
Assume that there exist $\gamma_0>0$ and $p>32$ such that $\EE{\exp(\gamma_0\|X_0\|_{L^2}^2)} < \infty$ and $\EE{\|(-A)^{\frac{1}{2}}X_0\|_{L^2}^p} < \infty$. Let $K_{\max} \in (0,\infty)$ and let $Q_1,Q_2 \in \mathcal{L}_1(L^2)$ be positive and self-adjoint, such that $\max\{\tr(Q_1),\tr(Q_2)\} \leq K_{\max}$. Let $\varphi:L^2 \to \R$ be twice Fréchet differentiable with continuous and bounded first and second order derivatives. Then, for all $\epsilon >0$ there exists $C_{\epsilon,\gamma_0,p,T}(X_0, K_{\max}, \varphi)\in (0,\infty)$ such that
    \begin{align}\label{eq:weak X^Q_1 vs X^Q_2 intro}
        \big|\mathbb{E}\big[\varphi(X_1(T))\big] - \mathbb{E}\big[\varphi(X_2(T))\big]\big| 
        \leq C_{\epsilon,\gamma_0,p,T}(X_0,K_{\max}, \varphi)
    \| (-A)^{-(1-\epsilon)} (Q_1-Q_2) \|_{\calL_1(L^2)},
    \end{align}
    and for all $\epsilon >0$, $r\in [1,\frac{p}{4})$  there exists $C_{\epsilon,\gamma_0,p,r,T}(X_0,K_{\max}) \in (0,\infty)$ such that 
    \begin{align}\label{eq:strong X^Q_1 vs X^Q_2 intro}
        \sup_{t\in [0,T]}\|X_1(t) - X_2(t)\|_{L^r(\Omega; L^2)} \leq C_{\epsilon, \gamma_0,p,r,T}(X_0,K_{\max})
        \big\| (-A)^{-(\nicefrac{1}{2}-\epsilon)} \big|Q_1^{\nicefrac{1}{2}} -Q_2^{\nicefrac{1}{2}}\big| \big\|_{\calL_2(L^2)}.
    \end{align}
\end{theorem}
Note that the weak bound is roughly the square of the strong bound. This aligns with the general principle in numerical analysis for stochastic differential equations that states that the weak convergence rate is twice the strong rate. We believe the bound in~\eqref{eq:strong X^Q_1 vs X^Q_2 intro} to be essentially sharp (see Remark~\ref{rem:sharp strong bound}).

The main motivation for our work lies in the numerical approximation of the Burgers equation. More specifically, both weak and strong convergence rates for the Galerkin approximation of~\eqref{eq:Burgers_intro} are known, see~\cite{Brehier_Cox_Millet_2024,Hutzenthaler_Jentzen_2020}. These Galerkin approximations lead to finite-dimensional systems driven by finite-dimensional noise, where the covariance operator of the noise is given by $P_M Q_1 P_M$ (where $P_M$ is $M^{\textnormal{th}}$ Galerkin projection, i.e., the projection onto the first $M$ eigenvectors of the Laplacian). However, sampling from a $P_M Q_1 P_M$-Brownian motion is not straightforward due to the infinite-dimensional nature of $Q_1$. This becomes easier if we can replace $Q_1$ by a finite-dimensional operator, and Theorem~\ref{thm:strong perturbation intro} quantifies the error introduced in doing so -- see also Corollaries~\ref{cor: weak conv X X^{Q_N} X^{Q_N}_M} and~\ref{cor: strong conv X X^{Q_N} X^{Q_N}_M} below.\par \medskip

The question of approximating the noise goes back at least as far as~\cite{Brzezniak:1997}, where in Section 5 it is proven that approximating the noise in a semilinear equation by a finite-dimensional noise results in a strongly convergent sequence of solutions. In~\cite{KunzeNeerven:2011,KunzeNeerven:2012} the authors establish that the solution to a semilinear equation depends continuously on both the drift and the diffusion coefficient. Neither~\cite{Brzezniak:1997} nor~\cite{KunzeNeerven:2011, KunzeNeerven:2012} provide quantitative results, the first result in that direction seems to be in~\cite{Harms_Muller_2019}, where both weak and strong bounds are provided for noise approximation for hyperbolic equations with globally Lipschitz coefficients. As the equations considered in~\cite{Harms_Muller_2019} do not involve an analytic operator in the leading term, the obtained rates do not involve (and benefit from) negative powers of $A$. Finally, a powerful abstract perturbation result for stochastic differential equations with monotone coefficients was obtained in~\cite{Hutzenthaler_Jentzen_2020}, indeed, this abstract result is the key ingredient for obtaining the strong error bound~\eqref{eq:strong X^Q_1 vs X^Q_2 intro}. The weak error bound~\eqref{eq:weak X^Q_1 vs X^Q_2 intro} relies on regularity results for the Kolmogorov equation associated with the Burgers equation established in~\cite{Brehier_Cox_Millet_2024}. For practical purposes, we first show that the bounds~\eqref{eq:weak X^Q_1 vs X^Q_2 intro} and~\eqref{eq:strong X^Q_1 vs X^Q_2} hold for the Galerkin approximations of the Burgers equation, with constants independent of the Galerkin projection. This allows us to pass to the limit and ultimately obtain~\eqref{eq:weak X^Q_1 vs X^Q_2 intro} and~\eqref{eq:strong X^Q_1 vs X^Q_2}.

\subsection*{Outline}
Section~\ref{chapter 1} contains the preliminaries: a discussion of the setting, (exponential) moment bounds for the solution to the stochastic Burgers equation~\eqref{eq:Burgers_intro}, and the relevant regularity results for the associated Kolmogorov equation. Both the moment bounds and the regularity results are taken from~\cite{Brehier_Cox_Millet_2024}. However, as the dependence on the covariance operator $Q$ of the involved constants is not tracked in~\cite{Brehier_Cox_Millet_2024}, the parts of the proof where this is relevant are provided in Appendix~\ref{App A}. The weak bounds are provided in Section~\ref{sec:weak bound}, and the strong bounds in Section~\ref{sec: strong bound}.

\section{Preliminaries} \label{chapter 1}

\subsection{Notation} 
Throughout this work, the following notation is used: let $\N = \{1,2,...\}$.\par 
Throughout this section we assume we are given two separable real Hilbert spaces $(H,\inn{\cdot,\cdot}_H,\|\cdot\|_H)$ and $(U,\inn{\cdot,\cdot}_U,\|\cdot\|_U)$ be separable real Hilbert spaces. \par 
Let $(\mathcal{L}(H,U), \|\cdot\|_{\mathcal{L}(H,U)})$ denote the Banach space of bounded linear operators from $H$ to $U$, and set $\mathcal{L}(H) := \mathcal{L}(H,H)$. Similarly, we let $(\calL(H\times H,U), \|\cdot\|_{\mathcal{L}(H\times H,U)} )$ denote the Banach space of bounded bilinear operators from $H\times H$ to $U$, endowed with the norm $\| B \|_{\mathcal{L}(H\times H,U)} = \sup_{g,h\in H, \|g\|_H = \|h\|_H =1} \| B(g,h) \|_U$. Let $K(H)$ denote the space of compact operators on $H$, and for all $p\in [1,\infty)$ let 
$\calL_{p}(H)\subseteq{K}(H)$ be the Banach space of Schatten class operators on $H$; $\| A \|_{\calL_p(H)}^{p} = \sum_{\lambda \in \sigma(A^*A)} \lambda^{p/2}$ (we assume the reader is familiar with the spectral theorem for compact self-adjoint operators).
In particular, $\calL_2(H)$ is the space of Hilbert-Schmidt operators on $H$ and $\calL_1(H)$ is the space of trace-class operators on $H$. Recall that $\calL_2(H)$ is a (separable)  Hilbert space under the inner product 
$\langle A, B \rangle_{\calL_2(H)} = \sum_{n=1}^{\infty} \langle A h_n, B h_n \rangle_{H}$, where $(h_n)_{n\in \N}$ is an orthonormal basis for $H$ and the inner product does not depend on the choice of the orthonormal basis. Also recall that we have, for all $p\in [1,\infty)$ and all $A\in \calL_p(H), B,C\in \calL(H)$, that 
\begin{equation}\label{eq:ideal}
\| BAC \|_{\calL_p(H)} \leq \| B \|_{\calL(H)} \| A \|_{\calL_p(H)} \| C\|_{\calL(H)}.
\end{equation}
In addition, we recall that the trace of $A\in \calL_1(H)$ is defined by
\begin{equation}
 \tr(A) = \sum_{n\in \N} \langle Ah_n, h_n \rangle_{H} \in  \mathbb{\R} ,
\end{equation}
where $(h_n)_{n\in \N}$ is an orthonormal basis for $H$ (and the value of $\tr(A)$ is independent of the choice of the orthonormal basis).
For $S \in \mathcal{L}(H,U)$, let $S^*\in \calL(U,H)$ denote the adjoint of $S$. An operator $S \in \mathcal{L}(H)$ is called \emph{self-adjoint} if $S = S^*$, and it is called \emph{positive} if $\inn{Sx,x}_H \geq 0$ for all $x \in H$.\par  

Given a measure space $(S,\Sigma,\mu)$ and $p\in [1,\infty]$, denote by $L^p(S;H)$ the Bochner space of measurable functions\footnote{See~\cite[Chapter 1]{Hytonen_vanNeerven_Veraar_Weis_2016} for the definition of Bochner spaces, note that the notions of strong and weak measurability coincide as we assume $H$ to be separable.}, which is a Banach space when endowed with the norm
\begin{align}
    \|f\|_{L^p(S;H)} &= \left(\int_S \|f\|^p_H\, d\mu \right)^{\frac{1}{p}}, \quad p\in [1,\infty);\label{def:Bochner norm}
    & 
    \|f\|_{L^{\infty}(S;H)} &= \operatorname{ess\, sup}_{s\in S} f(s). 
\end{align}
We set $L^p(S)=L^p(S;\R)$ and $L^p=L^p(0,1)$, $p\in [1,\infty]$. 
For $k \in \N$ and $p \in [1,\infty]$, let $W^{k,p} \subset L^p$ denote the Sobolev space on $(0,1)$, as defined in \cite{Evans_1998}. Denote the weak derivative by $\nabla$. Let $W^{k,p}_0 \subset W^{k,p}$ be the closure of $C^\infty_c((0,1))$ in $W^{k,p}$. For future reference we recall Poincar\'e's inequality: for all $x \in W^{1,2}_0$, it holds that
\begin{align}
    \Ltwo{x} \leq \frac{1}{\sqrt{2}}\Ltwo{\nabla x}. \label{eq:poincare} 
\end{align} \par 
Given an interval $I \subset \R$, let $C^\infty_c(I)$ be the space of infinitely often continuously differentiable functions from $I$ to $\R$ with compact support in $I$. Moreover, given a Banach space $(X,\|\cdot\|_X)$ and $\mu \in (0,1]$, let $C^\mu(I,X)$ be the Hölder space of functions $f:I\to X$ for which the Hölder norm
\begin{align*}
    \|f\|_{C^\mu(I,X)} = \sup_{x\in I} \|f(x)\|_X + \sup_{\substack{x,y\in I\\ x\ne y}}\frac{\|f(x)-f(y)\|_X}{|x-y|^\mu}
\end{align*}
is finite.\par
Given a twice Fréchet differentiable function $f:H\to \R$, denote its first and second order derivatives at $x \in H$ by $Df(x)\colon H \ni h \mapsto Df(x).(h)$ and $D^2f(x)\colon H\times H \ni (g,h) \mapsto D^2f(x).(g,h)$.

\subsection{The setting for the 1D stochastic Burgers equation}
\label{ssec:setting}
The formal setting we use throughout this paper to describe the 1D stochastic Burgers equation with trace class noise is taken from~\cite[Section 2.3]{Brehier_Cox_Millet_2024}. We repeat this setting here for the reader's convenience. \par 

We let $T\in(0,\infty)$ denote the terminal time.

We let $A\colon W^{1,2}_0 \cap W^{2,2} \subseteq L^2 \rightarrow L^2$ denote the Dirichlet Laplace operator on $L^2$, i.e.,
\begin{equation}\label{eq:defLaplace}
A x = - \sum_{k\in \N} (\pi k)^2 \langle x, h_k\rangle_{L^2} h_k,\quad 
\forall x\in W^{2,2}\cap W^{1,2}_0,
\end{equation}
where $h_k=\sqrt{2}\sin(k\pi \cdot)$ for all $k\in\N$. The eigenvectors $\big(h_k\big)_{k\in\N}$ define a complete orthonormal system of $L^2$. Note that $A$ generates an analytic $C_0$-semigroup $(e^{tA})_{t\geq 0}$ on $L^2$. 

We let $B\colon W^{1,2}\times W^{1,2}\rightarrow L^1$ denote the bilinear operator defined by
\begin{equation}
    B[x_1,x_2]=x_1 \nabla x_2 + x_2 \nabla x_1,\quad \forall x_1,x_2\in W^{1,2}, 
\end{equation}
and we set $B(x)=B[x,x]$ for $x\in W^{1,2}$. Note that an integration by parts yields the identity $\langle B(x),x\rangle_{L^2} =0$ for all $x\in W^{1,2}$.

We fix a filtered probability space $(\Omega,\calF,\Prob,(\calF_t)_{t\in [0,T]})$, which is assumed to be large enough to allow for the existence of a sequence of independent standard $(\calF_t)_{t\in [0,T]}$-Brownian motions $(W^{(k)})_{k\in \N}$. Let $(\tilde{h}_k)_{k\in \N}$ be an orthonormal basis for $H$ (the choice of $(\tilde{h}_k)$ is irrelevant). Given a positive self-adjoint $Q\in \calL_1(L^2)$, we define the $Q$-Brownian motion $W^{Q}\colon [0,T]\rightarrow H$ to be given by
\begin{equation}\label{eq:defWQ}
W^{Q}(t) = \sum_{k\in \N}  W^{(k)}(t) Q^{\nicefrac{1}{2}} \tilde{h}_k
\end{equation}
(note that the distribution of $W^Q$ is independent of the choice of $(\tilde{h}_k)_{k\in \N}$). 

 It follows from~\cite[Theorem 1.1 and Remark 3.1]{Liu_Rockner_2010} (see also~\cite[Section 2.3]{Brehier_Cox_Millet_2024}) that for every $p\in [4,\infty)$, every positive self-adjoint $Q\in \calL_1(L^2)$, and every $\calF_0$-measurable $X_0 \in L^p(\Omega,L^2)$ there exists a unique continuous (up to versions) $(\calF_t)_{t\in [0,T]}$-adapted process $X^{Q}\colon [0,T]\times \Omega \rightarrow H$ such that $\Prob(X^{Q}(t)\in W^{1,2}_0)=1$ for all $t\in [0,T]$,
\begin{equation*}
\E\left[
 \sup_{t\in [0,T]} \|X^{Q}(t)\|_{L^2}^{p}+
\int_{0}^{T}\|\nabla X^{Q}(t)\|_{L^2}^2 \,dt\,
\right]<\infty,
\end{equation*}
and 
\begin{equation}\label{eq:Burgers}
X^{Q}(t) = X_0 + \int_0^{t} \big[ A X^{Q}(s) + B(X^{Q}(s)) \big]\,ds + W^{Q}(t), \quad \forall t\in [0,T].
\end{equation}

We refer to the process $X^{Q}$ as the \emph{solution to the Burgers equation driven by a $Q$-Brownian motion (with initial value $X_0$)}.\par

As explained in the introduction, our proofs rely heavily on the spectral Galerkin approximations of~\eqref{eq:Burgers}, which we shall now introduce. For $M\in\N$, set
\[
H_M = \operatorname{span}(\{h_1,\ldots,h_M\})\subseteq W^{1,2}_{0} \cap W^{2,2}
\]
and let
 $P_M\in \mathcal{L}(L^2)$ denote the orthogonal projection onto $H_M$.
 Define the linear operator $A_M\in \mathcal{L}(L^2,H_M)$ and the bilinear operator $B_M\colon L^2 \times L^2 \rightarrow H_M $
by $A_M = P_M A P_M $ ($=A P_M$) and 
\begin{equation}\label{eq:defB_M}
B_M[x_1,x_2]= P_M B[P_Mx_1, P_Mx_2], \quad \forall x_1,x_2 \in L^2.
\end{equation}
We set $B_M(x)= B_M[x,x]$ for all $x\in L^2$. Note that for $x_1,x_2,y\in H_M$ we have
\begin{align}
    \innt{y,B_M[x_1,x_2]} = \innt{y,B[x_1,x_2]} \label{B=B_M with inner products}.
\end{align}
It again follows from~\cite[Remark 3.1 and Theorem 1.1]{Liu_Rockner_2010} (see also~\cite[Section 2.3]{Brehier_Cox_Millet_2024}) that for every $p\in [4,\infty)$, every positive self-adjoint $Q\in \calL_1(L^2)$, and every $\calF_0$-measurable $X_0 \in L^p(\Omega,L^2)$ there exists a unique (up to versions) stochastic process $X_M^{Q}\colon [0,T]\times \Omega \rightarrow H_M$ such that
\begin{equation}\label{eq:Galerkin approx}
    X_M^{Q}(t) = P_M X_0 + 
    \int_{0}^{t}\big[ A_M X_M^{Q}(s) + B_M(X_M^{Q}(s))\big]\,ds + P_M W^Q(t),\quad \forall  t\in [0,T].
\end{equation} 
Moreover, the solution $X_M^{Q}$ can be written using the following  mild formulation: 
\begin{equation}\label{eq:mildGalsolBurgers}
    X_M^{Q}(t) = e^{tA} P_M X_0 + 
    \int_{0}^{t} e^{(t-s)A} B_M(X_M^{Q}(s))\,ds + 
    \int_0^t e^{(t-s)A} P_M\,dW^Q(s),\quad \forall  t\in [0,T].
\end{equation}
We refer to the processes $X_M^{Q}$, $M\in \N$, as the \emph{Galerkin approximations of $X^{Q}$}.

\subsection{Properties of the operators $A$ and $B$}
We list some frequently used properties of the operators $A, A_M, B$ and $B_M$, as introduced in Section~\ref{ssec:setting}.\par 
For any $\alpha \in [0,\infty)$ we can define the fractional powers $(-A)^{\alpha} \colon D((-A)^{\alpha}) \subseteq L^2 \rightarrow L^2$ of $-A$, by
\begin{equation*}
D((-A)^{\alpha}) = \Big\{ x \in L^2 \colon  \sum_{k\in \N} (\pi k)^{4\alpha} \langle x, h_k\rangle_{L^2}^2 < \infty\Big\} 
\end{equation*}
and
\begin{equation}\label{eq:def_fracpowA}
    (-A)^{\alpha}x = \sum_{k\in \N} (\pi k)^{2\alpha} \langle x, h_k\rangle_{L^2} h_k, \quad \forall  x\in D((-A)^{\alpha}).
\end{equation}
Furthermore, we define $(-A)^{-\alpha} \in \mathcal{L}(((-A)^{\alpha})^*,L^2)$ to be the adjoint of $A^{\alpha}$. Parseval's identity implies that
\begin{equation}\label{eq:A_fracpownorminc}
\| (-A)^{\alpha} x \|_{L^2} \leq  \| (-A)^{\beta} x \|_{L^2}
\end{equation} 
for all $\alpha,\beta\in (-\infty,\infty)$ satisfying $\alpha<\beta$ and all $x\in  D((-A)^{\beta\vee 0})$. \par 
Note that $(-A)^{-\alpha} \in \calL_1(L^2)$ and $(-A)^{-\nicefrac{\alpha}{2}} \in \calL_2(L^2)$ if and only if $\alpha > \frac{1}{2}$. Indeed,
\begin{align}
    \|(-A)^{-\nicefrac{\alpha}{2}}\|^2_{\calL_2(L^2)} = \|(-A)^{-\alpha}\|_{\calL_1(L^2)} = \sum_{k \in \N}\frac{1}{(\pi k)^{2\alpha}} < \infty \iff \alpha > \frac{1}{2}. \label{eq: A L_1 reg}
\end{align}
We recall the so-called smoothing property of the semigroup $\bigl(e^{tA}\bigr)_{t\ge 0}$, see e.g.~\cite[Chapter 2.6]{Pazy:1983}. In our setting the proof is elementary (see also~\cite[Lemma 2.2]{Brehier_Cox_Millet_2024}.

\begin{lemma}\label{lem:A_analytic}
For all $\alpha \in (0,\infty) $, $t\in (0,\infty)$, and $x\in L^2$ one has
\begin{equation}
\| (-A)^{\alpha}  e^{t A} x \|_{L^2} \leq e^{\alpha(\log(\alpha)-1)} t^{-\alpha} \| x \|_{L^2}.
\end{equation}
\end{lemma}
The following simple version of the Sobolev embedding is rather straightforward to prove:
\begin{proposition}
    Let $\delta >0$. Then there exists $C_\delta>0$ such that 
    \begin{align}
        \|x\|_{L^\infty} \leq C_\delta\|(-A)^{\frac{1+\delta}{4}}x\|_{L^2}, \quad \forall  x \in D((-A)^{\frac{1+\delta}{4}}).\label{eq:Linf bound by L2}
    \end{align}
\end{proposition}
\begin{proof}
Recalling $|h_k(z)|\leq \sqrt{2}$ for all $z \in (0,1)$ and applying the Cauchy-Schwarz inequality we obtain
    \begin{align*}
        \|x\|_{L^\infty} &= \sup_{z \in [0,1]}\left|\sum_{k\in \N} \innt{x,h_k}h_k(z)\right| \leq \sqrt{2}\sum_{k\in \N} |\langle x,h_k\rangle_{L^2}| \\
        & \leq \sqrt{2}\left(\sum_{k\in \N} (\pi k)^{-(1+\delta)}\right)^{\frac{1}{2}}
        \left(\sum_{k \in \N}(\pi k)^{1+\delta}|\innt{x,h_k}|^2\right)^{\frac{1}{2}}
        = \sqrt{2}\left(\sum_{k\in \N} (\pi k)^{-(1+\delta)}\right)^{\frac{1}{2}} \| (-A)^{\frac{1+\delta}{4}} x \|_{L^2}.
    \end{align*}
\end{proof}
As for the operator $B_M$, using~\eqref{B=B_M with inner products} and integrating by parts we obtain that for $x,y \in H_M$ that 
    \begin{align}\label{eq:B_cancelations}
        \innt{x,B_M[x,y]} = \innt{x,B[x,y]} = -\tfrac{1}{2} \innt{y,B(x)} = -\tfrac{1}{2}\innt{y,B_M(x)}.
    \end{align} 

\subsection{Estimates for stochastic convolutions}
In this section we provide some bounds on the stochastic convolution with the semigroup $(e^{tA})_{t\in [0,\infty)}$ that we need to prove strong convergence rates. 
Let $Q\in \calL_1(L^2)$ be positive and self-adjoint. Consider
\begin{align}
    Y_M(t) = \int_0^t e^{(t-s)A}P_M\,dW^{Q}(s),\quad t\in [0,T]. \label{eq:stoch conv perturbation}
\end{align} 
Then $Y_M \in H_M$ and satisfies the evolution equation
\begin{align}
    Y_M(t) = \int_0^t AY_M(s)\, ds + P_M W^{Q}(t). \label{eq:Y_M evolution eq}
\end{align}
We next provide some regularity estimates for $Y_M$: 
\begin{proposition}
     \label{prop:stoch conv moment bound}
    Let $p > 0$ and $\alpha \in [0,\frac{1}{2})$. \!Then there exists $C_{\alpha,p} \in (0,\infty)$ such that for all $t\in [0,T], M \in \N$ it holds that 
    \begin{align*}
        \EE{\Ltwo{ Y_M(t)}^p} \leq C_{\alpha,p} T^{\nicefrac{(1-2\alpha)p}{2}}
        \| (-A)^{-\alpha}  Q^{\nicefrac{1}{2}} \|_{\calL_2(L^2)}^p.
    \end{align*}
\end{proposition}

\begin{proof}
    From~\eqref{eq:stoch conv perturbation}, the Burkholder-Davis-Gundy inequality, the ideal property of $\calL_2(L^2)$ (see~\eqref{eq:ideal}), and Lemma~\ref{lem:A_analytic} we obtain (for $\alpha\in (0,\frac{1}{2})$):
\begin{align*}
    \EE{\Ltwo{ Y_M(t)}^p} 
    & \leq C_{p}
    \left|
        \int_0^{t} \| e^{(t-s)A} P_M Q^{\nicefrac{1}{2}} \|_{\calL_2(L^2)}^2 \,ds 
    \right|^{\nicefrac{p}{2}}
    \\ & \leq C_{p} e^{\alpha p (\log(\alpha) -1)}
    \left|
        \int_0^{t} (t-s)^{-2\alpha}  \,ds 
    \right|^{\nicefrac{p}{2}} \| (-A)^{-\alpha} P_M Q^{\nicefrac{1}{2}} \|_{\calL_2(L^2)}^{p}.
\end{align*}
For $\alpha\in (0,\frac{1}{2})$, the result now follows by observing that 
\begin{equation}
\| (-A)^{-\alpha} P_M Q^{\nicefrac{1}{2}} \|_{\calL_2(L^2)}
= \| P_M (-A)^{-\alpha}  Q^{\nicefrac{1}{2}} \|_{\calL_2(L^2)} 
\leq \| (-A)^{-\alpha}  Q^{\nicefrac{1}{2}} \|_{\calL_2(L^2)},
\end{equation}
where again we used the ideal property and the fact that $\|P_M\|_{\mathcal{L}(L^2)} = 1$. For $\alpha=0$ the reasoning is analogous, using $\| e^{tA} \|_{\calL(L^2)}\leq 1$.
\end{proof}

We will also need an $L^\infty$ bound for $Y_M$:
\begin{lemma} \label{lem:Y Linf bound}
    Let $p \geq 1$. Then there exists $C_{p,T} \in (0,\infty)$ such that for all $M \in \N$, it holds that 
    \begin{align}
    \EE{\sup_{t\in [0,T]}\|Y_M(t)\|^p_{L^\infty}} \leq C_{p,T}\tr(Q)^{\nicefrac{p}{2}}. \label{eq:Y Linf}
\end{align}
\end{lemma}
\begin{proof}
     By \eqref{eq:Linf bound by L2}, for all $\lambda >\frac{1}{4}$ there exists $C_{\lambda}$ such that
\begin{align*}
    \|Y_M(t)\|_{L^\infty} \leq C_{\lambda}\|(-A)^{\lambda}Y_M(t)\|_{L^2},
\end{align*}
and hence for all $\mu > 0$ one has
\begin{align}
    \sup_{t\in [0,T]}\|Y_M(t)\|_{L^\infty} \leq C_{\lambda,\mu}\|(-A)^{\lambda}Y_M\|_{C^{\mu}([0,T],L^2)}. \label{eq:Linf ineq 1}
\end{align}
Moreover, for all $M\in \N$, $\lambda,\mu > 0$ with $\lambda +\mu < \frac{1}{2}$ and $p \geq 1$, \cite[Lemma A.1]{Brehier_Cox_Millet_2024} states that there exists $C_{\lambda,\mu,p,T} > 0$ such that
\begin{align}
     \EE{\|(-A)^{\lambda}Y_M\|^{p}_{C^{\mu}([0,T],L^2)}}\leq C_{\lambda,\mu,p,T}\tr (Q)^{\nicefrac{p}{2}}. \label{eq:Linf ineq 2}
\end{align}
Choosing, say, $\lambda = \frac{5}{16}, \mu = \frac{1}{16}$ for both inequalities \eqref{eq:Linf ineq 1} and \eqref{eq:Linf ineq 2} and combining them finishes the proof.
\end{proof}
In addition, we have the following exponential estimate:
\begin{lemma} \label{lem:Y exp bound}
    For all $\alpha \in \left(0,\frac{1}{2\|Q\|_{\mathcal{L}(L^2)}}\right), M\in \N$, it holds that
\begin{align}
    \EE{\exp\left(\alpha \sup_{t\in [0,T]}\Ltwo{Y_M(t)}^2+\alpha\int_0^T \Ltwo{\nabla Y_M(s)}^2 ds\right)} &\leq  2e^{\alpha T \tr (Q)}. \label{eq:Y exp bound}
\end{align}
\end{lemma}
See Appendix~\ref{App A} for a proof.

\subsection{Moment bounds}\label{ssec:bounds}
In this section we present some relevant moment bounds for $X_M^Q$, i.e., for the Galerkin approximation of the solution to the Burgers equation driven by a $Q$-Brownian motion (see Section~\ref{ssec:setting}, in particular equation~\eqref{eq:Galerkin approx}) These moment bounds have been established in \cite[Section 3]{Brehier_Cox_Millet_2024}. The only difference between the results stated here and those in \cite[Section 3]{Brehier_Cox_Millet_2024} is that we are more precise about the dependence on $Q$ for the constants involved, which is crucial in our setting because we need to consider bounds for both $X^{Q_1}_M$ and $X^{Q_2}_M$.\par  
Here, we only state the results. In Appendix~\ref{App A} we provide proofs for those bounds where the dependency on $Q$ is not made explicit in~\cite[Appendix A]{Brehier_Cox_Millet_2024}.
\begin{lemma} \label{lem: X_M sup +exp moment bounds}
(i) Suppose there exists $p \geq 4$ such that $\E \left[\Ltwo{X_0}^p\right] < \infty$. Then there exists an increasing function $F_{p,T}:[0,\infty) \to (0,\infty)$ (dependent only on $p$ and $T$) such that for all positive self-adjoint $Q\in \calL_1(L^2)$ we have
\begin{align}
    \sup_{M\in \N} \E\left[\sup_{t \in [0,T]}\bigLtwo{X_M^Q(t)}^p + 2p\int_0^T \bigLtwo{X_M^Q(t)}^{p-2}\bigLtwo{\nabla X_M^Q(t)}^2\,dt\right] \leq F_{p,T}(\tr(Q))\left(\E\left[\Ltwo{X_0}^p\right]+1\right) .\label{eq:X_M p moment bound}
\end{align}
(ii) If there exists $\gamma_0 >0$ such that $\E\left[\exp (\gamma_0\|X_0\|_{L^2}^2)\right] < \infty$, then for all positive self-adjoint $Q\in \calL_1(L^2)$ and all $\beta \in \left(0,\frac{\gamma_0}{1+2\gamma_0\|Q\|_{\mathcal{L}(L^2)}}\right)$ we have
\begin{align}
     \sup_{M\in \N} \mathbb{E}\bigg[\exp\Big(\beta\sup_{t \in [0,T]}\|X_M^Q(t)\|_{L^2}^2 + \beta \int_0^T \|\nabla X_M^Q(t)\|_{L^2}^2 \,dt\Big)\bigg]\leq 2e^{\beta T  \tr(Q)}\E\left[\exp(\gamma_0\|X_0\|_{L^2}^2)\right]^{\frac{\beta}{\gamma_0}}. \label{eq:X_M exp bound random initial}
\end{align}
In particular, if $X_0 = x_0 \in L^2$ is a deterministic initial condition, \!then for all positive self-adjoint $Q\in \calL_1(L^2)$, $\beta \in \left(0,\frac{1}{2\|Q\|_{\mathcal{L}(L^2)}}\right)$,
\begin{align}
     \sup_{M\in \N} \mathbb{E}\bigg[\exp\bigg(\beta\sup_{t \in [0,T]}\bigLtwo{X_M^Q(t)}^2 + \beta \int_0^T \bigLtwo{\nabla X_M^Q(t)}^2\,dt\bigg)\bigg] \leq 2e^{\beta T  \tr(Q)}\exp(\beta \Ltwo{x_0}^2). \label{eq:X_M exp bound deterministic initial}
\end{align}    
\end{lemma}

\begin{lemma} \label{lem:L_inf bound of X_M}
    Let $\alpha \in (\frac{1}{4}, \frac{1}{2})$ and $p \geq  \frac{8}{3}$ be such that the initial condition $X_0 \in D((-A)^{\alpha})$ a.s. and $\E\left[\Ltwo{(-A)^{\alpha}X_0}^p\right] + \E\left[\Ltwo{X_0}^{3p}\right] < \infty$. Then there is a non-decreasing  $F_{p,T}\colon [0,\infty) \to  (0,\infty)$ (dependent on $p$ and $T$) and $C_{p,\alpha,T} >0$ such that for all positive self-adjoint $Q\in \calL_1(L^2)$ we have
\begin{align}
    \sup_{M \in \N} \mathbb{E}\bigg[\sup_{t \in [0,T], z \in [0,1]}|X_M^Q(t)(z)|^p\bigg] \leq C_{p,\alpha,T}F_{p,T}(\tr(Q))\left(1+\EE{\Ltwo{(-A)^{\alpha}X_0}^p} + \EE{\Ltwo{X_0}^{3p}}\right).
\end{align}
\end{lemma}

\begin{lemma} \label{lem:(-A)X_M bound}
    Let $\alpha \in (\frac{1}{4}, \frac{1}{2})$ and $p \geq  \frac{4}{3}$ be such that the initial condition $X_0 \in D((-A)^{\alpha})$ a.s. and $\EE{\Ltwo{(-A)^{\alpha}X_0}^{2p}} + \EE{\Ltwo{X_0}^{6p}} < \infty$. Then for all $\lambda,\gamma \in \left(0,\frac{1}{2}\right)$ with $\lambda + \gamma < \alpha$, there exists a non-decreasing $F_{p,T}:[0,\infty)\to (0,\infty)$ (dependent on $p$ and $T$) and $C_{p,\alpha,\gamma,\lambda,T} >0$ such that for all positive self-adjoint $Q\in \calL_1(L^2)$ we have 
\begin{align}
    \sup_{M \in \N} \EE{\|(-A)^{\lambda}X_M^Q\|^p_{C^\gamma([0,T],L^2)}} \leq C_{p,\alpha,\gamma,\lambda,T}F_{p,T}(\tr (Q))\left(1+\mathbb{E}[\Ltwo{(-A)^{\alpha}X_0}^{2p}] + \mathbb{E}[\Ltwo{X_0}^{6p}]\right). \label{eq:(-A)X_M inequality}
\end{align}
\end{lemma} 

Based on the results above, we now provide some very specific bounds needed to establish the weak error bounds. We begin by introducing, for $q \geq 1$ and $\delta, \epsilon >0$, the function $\Psi_{\delta,\epsilon,q}\colon D((-A)^{\frac{1}{4}+\delta}) \rightarrow \R$ given by
\begin{align} \label{def:Psi}
    \Psi_{\delta,\epsilon,q}(x): = e^{\epsilon \Ltwo{x}^2}\left(1+\|(-A)^{\frac{1}{4}+\delta}x\|_{L^2}^q\right),\quad x \in D((-A)^{\frac{1}{4}+\delta}).
\end{align} 
By slightly adapting \cite[Lemma 5.4]{Brehier_Cox_Millet_2024}, we have the following bound for $\Psi_{\delta,\epsilon,q}(X_M^Q(t))$:
\begin{lemma} \label{lem:sup of psi}
    Suppose there exist $\gamma_0 >0$, $p>32$ and $\delta_0 >0$ such that $\EE{\exp(\gamma_0\|X_0\|_{L^2}^2)} < \infty$ and $\EE{\|(-A)^{\frac{1}{4}+\delta_0}X_0\|_{L^2}^p} < \infty$. Then, for all $q\in [1,\frac{p}{2})$, and $\delta \in [0,\delta_0)$ there exists a non-decreasing function $F_{p,T}:[0,\infty) \to (0,\infty)$ (dependent on $p$ and $T$) and $C_{\gamma_0,\delta,\delta_0,p,q,T}>0$ such that for all positive self-adjoint $Q\in \calL_1(L^2)$ and all $\epsilon \in (0,\frac{(p-2q)\gamma_0}{p(1+2\gamma_0\|Q\|_{\mathcal{L}(L^2)})})$ we have
    \begin{multline}
        \sup_{M\in \N} 
        \EE{\sup_{t\in [0,T]}\Psi_{\delta,\epsilon,q}(X_M^Q(t))} \\
        \leq C_{\gamma_0,\delta,\delta_0,p,q,T}F_{p,T}(\tr(Q))\left(1+\EE{\exp(\gamma_0 \|X_0\|_{L^2}^2)}\right)^2\left(1+\EE{\|(-A)^{\frac{1}{4}+\delta_0}X_0\|^{p}_{L^2}}\right)^{\frac{2q}{p}}. \label{eq:sup of Psi bound}
    \end{multline}
\end{lemma}

\subsection{The Kolmogorov equation}
This section is concerned with the \emph{Kolmogorov backward partial differential equation}, or simply \emph{Kolmogorov equation}, associated to $X^Q_M$, i.e., to the Galerkin approximation of the stochastic Burgers equation  driven by a $Q$-Brownian motion, see equation \eqref{eq:Galerkin approx}. \par 

Recall the setting of Section~\ref{ssec:setting}, in particular, recall the definitions of $H_M, P_M, A_M,$ and $B_M$. Fix a positive self-adjoint $Q \in \calL_1(L^2)$. Given $x\in H_M$, let $X_M^x(t)$ denote the Galerkin approximation of stochastic Burger's equation \eqref{eq:Galerkin approx} driven by a $Q$-Brownian motion, and with deterministic initial value $X_0 = x$. 

\begin{theorem}[Kolmogorov equation] \label{Kolm thm}
Let $\varphi:L^2(0,1) \to \R$ be a twice continuously Fréchet differentiable function with bounded first and second derivatives. Define $u^Q_M \colon [0,T]\times H_M \to \R$ by
\begin{align}
    u^Q_M(t,x) = \EE{\varphi(X^x_M(t))}. \label{def:u_M}
\end{align}
Then $u^Q_M(t,\cdot)$ has first and second order Fréchet derivatives for all $t\in [0,T]$; we denote these derivatives by $Du^Q_M(t,x)$ and $D^2u^Q_M(t,x)$, respectively. 
Moreover, $u^Q_M$, $Du^Q_M$ and $D^2u^Q_M$ are continuous in both variables, $u^Q_M(\cdot,x)$ is differentiable, and $u^Q_M$ satisfies the following differential equation:
    \begin{equation}
    \begin{cases}
        \frac{\partial u^Q_M}{\partial t}(t,x) =  Du^Q_M(t,x).(Ax + B_M(x)) 
 + \frac{1}{2} \tr(D^2u^Q_M(t,x) P_M Q P_M)
\label{eq:Kolmogorov}, \\[0.5em]
         u^Q_M(0,x) = \varphi(x). 
    \end{cases}
    \end{equation}
\end{theorem}

\begin{remark} \label{diffb of u_M}
The proof of Theorem~\ref{Kolm thm} is beyond the scope of this paper, but we will sketch a possible approach. Indeed, a proof of this result for S(P)DEs with coefficients that have bounded derivatives up to third order can be found e.g. in~\cite[Chapter 7]{DaPrato_Zabczyk_2002} (see Theorems 7.3.6, 7.4.5, and 7.5.1). While these results are not directly applicable to our setting ($B_M$ does not have a bounded first derivative), the approach in \cite{DaPrato_Zabczyk_2002} can be extended to our setting (also using the bounds presented in Section~\ref{ssec:bounds}).\par 
Indeed, following the approach in \cite{DaPrato_Zabczyk_2002}, one first establishes the Fréchet differentiability of the map $H_M \ni x \mapsto X^x_M \in H^p([0,T])$, where $p\geq 1$ and $H^p([0,T])$ is a space of progressively measurable $H_M$-valued processes on $[0,T]$, which is a Banach space when endowed with the norm 
    \begin{align*}
        \|Y\|_{H^p([0,T])} = \sup_{t\in [0,T]}\mathbb{E}[\|Y(t)\|^p_{L^2}]^{\frac{1}{p}}.
    \end{align*} 
To this end, one considers $F:H_M\times H^p([0,T]) \to H^p([0,T])$ defined by 
    \begin{align*}
        F(x,Y)(t) = x + \int_0^t A_MY(s) + B_M(Y(s)) \, ds + P_MW^Q(t) - Y(t).
    \end{align*}
Note that for all $x \in H_M$, $F(x,X^x_M) = 0$. Then, by a suitable implicit function theorem, $X^x_M = f(x)$ for some differentiable $f:H_M \to H^p([0,T])$, moreover, using the explicit expressions for $Df$ and $D^2f$, one obtains the dynamics\footnote{Regarding $\eta$ and $\zeta$: we drop the dependence on $Q$ as we only deal with these concepts within this section.} of $\eta^{h,x}_M:=Df(x).(h)$ and $\zeta^{g,h,x}_M:=D^2f(x).(g,h) $:
\begin{align} 
    \eta^{h,x}_M(t) = h + \int_0^t A\eta^{h,x}_M(s) + 2B_M[X^x_M(s),\eta^{h,x}_M(s)]\,ds, \quad h,x\in H_M,\label{eq:eta evolution}
\end{align}
and
\begin{align}
    \zeta^{g,h,x}_M(t) = \int_0^t A\zeta^{g,h,x}_M(s) + 2B_M[X^x_M(s),\zeta^{g,h,x}_M(s)] + 2B_M[\eta^{g,x}_M(s),\eta^{h,x}_M(s)]\,ds, \quad g,h,x\in H_M.\label{eq:zeta evolution}
\end{align}
Subsequently, one obtains Fr\'echet differentiability of $u^Q_M(t,\cdot)$ by observing that
\begin{align} \label{eq:Du expr}
    Du^Q_M(t,x).(h) = \EE{D\varphi(X^x_M(t)).(\eta^{h,x}_M(t))},\quad h,x\in H_M,
\end{align}
and 
\begin{align} \label{eq:D^2u expr}
    D^2u^Q_M(t,x).(g,h) = \EE{D\varphi(X^x_M(t)).(\zeta^{g,h,x}_M(t))} + \EE{D^2\varphi(X^x_M(t)).(\eta^{g,x}_M(t),\eta^{h,x}_M(t))}, \quad g,h,x\in H_M.
\end{align}
The continuity of $u^Q_M$, $Du^Q_M$, and $D^2u^Q_M(\cdot,x)$ follows from the representation of $X_M$, $\eta^{h,x}_M$, and $\zeta^{g,h,x}_M$ as It\^o processes, combined with the bounds from Section~\ref{ssec:bounds}. Finally, It\^o's formula allows one to conclude that $u^Q_M$ satisfies~\eqref{eq:Kolmogorov}. Some further details can also be found in~\cite[Chapter 3]{Urbonas:2025}.
\end{remark}

\subsection{Estimates for the solution of the Kolmogorov equation and related bounds} \label{Kolm section}

We have the following (uniform in $M$) estimate for $D^2u^Q_M(t,x)$:
\begin{lemma}[{\hspace{1sp}\cite[(Theorem 4.1 (ii)]{Brehier_Cox_Millet_2024}}] \label{lem:D^2u bound}
    Let $\alpha,\beta\in [0,1)$ be such that $\alpha+\beta<1$, and let $\delta, \epsilon\in (0,\infty)$ be arbitrarily small auxiliary parameters. Then for all $M\in \N$, $x\in H_M$ the mapping $(-A)^{\beta} D^2 u^Q_M(t,x)(-A)^{\alpha}$ extends to a bounded operator on $(H_M,\|\cdot \|_{L^2})$ and there exists $C_{\alpha,\beta,\delta,\epsilon,T}(Q,\varphi)\in (0,\infty)$ such that for all $M\in \N$, $x\in H_M$ and $t\in (0,T]$ one has that
\begin{equation}
\begin{aligned}\label{eq:Kol_2nd_deriv_est}
\|(-A)^{\beta} D^2 u_M^{Q}(t,x)(-A)^{\alpha}\|_{\calL(H_M)} 
& \le C_{\alpha,\beta,\delta,\epsilon,T}(Q,\varphi)\, t^{-(\alpha + \beta)}\,   e^{\epsilon \|x\|_{L^2}^2}\,\big( 1+\|(-A)^{\frac{1}{4}+\delta }x\|_{L^2}^{16}\big).  
\end{aligned}
\end{equation}
\end{lemma}

\begin{remark} \label{rem:dependence in Q}
 In fact, the constant $C_{\alpha,\beta,\delta, \epsilon,T}(Q,\varphi)$ above may be chosen such that $C_{\alpha,\beta,\delta, \epsilon,T}(Q,\varphi) = C_{\alpha,\beta,\delta, \epsilon,T}(\varphi)F_{T}(\tr(Q))$, where $F_{T}:[0,\infty)\to (0,\infty)$ is non-decreasing. This is because the dependence in $Q$ in the proofs of \cite{Brehier_Cox_Millet_2024} come from Lemmas proven in \cite[Section 3]{Brehier_Cox_Millet_2024}, which we revise in Section~\ref{ssec:bounds} to have more precise dependencies in $Q$.
\end{remark}

\section{Weak convergence}\label{sec:weak bound}
In this section we establish upper bounds for the quantity
\begin{equation}
|\E[\varphi(X^{Q_1})] - \E[\varphi(X^{Q_2})]|,
\end{equation}
where $X^{Q_1}$ and $X^{Q_2}$ are the solutions to~\eqref{eq:Burgers} with $Q=Q_1$ and $Q=Q_2$, $Q_1, Q_2\in \calL_1(L^2)$ are positive self-adjoint, and $\varphi\in C^2(L^2,\R)$ is sufficiently nice; see Corollary~\ref{cor:weak conv X X^{Q_2} X^{Q_2}_M} below. These bounds are obtained by first proving analogous, dimension-independent bounds for the Galerkin approximations, see Theorem~\ref{thm:weak perturbation} below. By combining Theorem~\ref{thm:weak perturbation} with the weak convergence rates for spectral Galerkin approximations of the stochastic Burgers equation recently obtained in~\cite{Brehier_Cox_Millet_2024}, we obtain the desired bounds. Note that this approach also allows us to establish weak convergence rates for an approximation obtained by applying a spectral Galerkin approximation in the spatial parameter and considering a Karhunen-Lo\'eve approximation of the noise; see Corollary~\ref{cor: weak conv X X^{Q_N} X^{Q_N}_M}.

\begin{theorem} \label{thm:weak perturbation}
    Suppose there exist $\gamma_0 >0$, $p>32$ and $\delta_0 >0$ such that $\EE{\exp(\gamma_0\|X_0\|_{L^2}^2)} < \infty$ and $\EE{\|(-A)^{\frac{1}{4}+\delta_0}X_0\|_{L^2}^p} < \infty$.  Let $K_{\max} \in (0,\infty)$ and let $Q_1,Q_2 \in \mathcal{L}_1(L^2)$ be positive and self-adjoint, such that $\max\{\tr(Q_1),\tr(Q_2)\} \leq K_{\max}$ and let $X^{Q_1}_M$, $X^{Q_2}_M$ ($M\in \N$) be solutions to~\eqref{eq:Galerkin approx} with $Q=Q_1$ and $Q=Q_2$. Then for all $\alpha \in [0,\min\{1,\frac{3}{4}+\delta_0\})$ there exists $C_{\alpha,\gamma_0,\delta_0,p,T}(X_0,K_{\max},\varphi) \in (0,\infty)$ such that
    \begin{align*}
        \sup_{M\in \N}\big|\mathbb{E}\big[\varphi(X^{Q_1}_M(T))\big] - \mathbb{E}\big[\varphi(X^{Q_2}_M(T))\big]\big| 
        \leq C_{\alpha,\gamma_0,\delta_0,p,T}(X_0,K_{\max},\varphi)
        \| (-A)^{-\alpha} (Q_1-Q_2) \|_{\calL_1(L^2)}.
    \end{align*}
\end{theorem}

We postpone the proof of this result to the end of this section. In order to obtain bounds when $M\rightarrow \infty$, we recall the following:

\begin{theorem}[{\hspace{1sp}\cite[Theorem 5.1]{Brehier_Cox_Millet_2024}}] \label{thm: weak rates Galerkin approx}
\!Assume there exist $\gamma_0,\delta_0>0$ and $p>32$\! such that $\EE{\exp(\gamma_0\|X_0\|_{L^2}^2)} < \infty$ and \mbox{$\EE{\|(-A)^{\frac{1}{4}+\delta_0}X_0\|_{L^2}^p} < \infty$.} Let $Q\in \mathcal{L}_1(L^2)$ be positive and self-adjoint and let $X^{Q}$, $X^{Q}_M$ ($M\in \N$) be solutions to~\eqref{eq:Burgers} and~\eqref{eq:Galerkin approx}. Let $\varphi:L^2 \to \R$ be twice Fréchet differentiable with bounded, continuous first and second order derivatives. Then, for all $\alpha \in [0,\min\{1,\frac{3}{4}+\delta_0\})$ there exists $C_{\alpha,\gamma_0,\delta_0,p,T}(Q,\varphi)\in (0,\infty)$ and an increasing function $F_{p,T}\colon [0,\infty)\rightarrow [1,\infty)$ such that for all $M\in \N$
\begin{multline}
    \big|\mathbb{E}\big[\varphi(X^Q(T))\big] - \mathbb{E}\big[\varphi(X_M^Q(T))\big]\big| \\ \leq C_{\alpha,\gamma_0,\delta_0,p,T}(\varphi) F_{p,T}(\tr(Q)) \left(1+\EE{\exp(\gamma_0 \|X_0\|_{L^2}^2)}\right)^2\left(1+\EE{\|(-A)^{\frac{1}{4}+\delta_0}X_0\|^{p}_{L^2}}\right)M^{-2\alpha}.\label{eq:weak conv X_M to X}
\end{multline}
\end{theorem}

\begin{remark}
Theorem 5.1 in \cite{Brehier_Cox_Millet_2024} is actually slightly less precise than the statement above: in \cite{Brehier_Cox_Millet_2024}, a bound is established involving a constant depending on $Q$, but it is not stated that this constant increases in $\tr(Q)$. However, in view of the more precise formulation of Lemmas~\ref{lem: X_M sup +exp moment bounds}-\ref{lem:(-A)X_M bound} (compared to~\cite[Lemmas 3.1-3.3]{Brehier_Cox_Millet_2024}), we can deduce the slightly stronger statement provided above. This is crucial for the following corollaries.
\end{remark}
Combining Theorems~\ref{thm:weak perturbation} and~\ref{thm: weak rates Galerkin approx} we obtain:
\begin{corollary}\label{cor:weak conv X X^{Q_2} X^{Q_2}_M} 
Assume that there exist $\gamma_0,\delta_0>0$ and $p>32$ such that $\EE{\exp(\gamma_0\|X_0\|_{L^2}^2)} < \infty$ and \mbox{$\EE{\|(-A)^{\frac{1}{4}+\delta_0}X_0\|_{L^2}^p} < \infty$.} Let $K_{\max} \in (0,\infty)$ and let $Q_1,Q_2 \in \mathcal{L}_1(L^2)$ be positive and self-adjoint, such that $\max\{\tr(Q_1),\tr(Q_2)\} \leq K_{\max}$ and let $X^{Q_1}$, $X^{Q_2}$, and $X^{Q_2}_M$ ($M\in \N$) be solutions to~\eqref{eq:Burgers} with $Q=Q_1$, to~\eqref{eq:Burgers} with $Q=Q_2$, and to~\eqref{eq:Galerkin approx} with $Q=Q_2$. Let $\varphi:L^2 \to \R$ be twice Fréchet differentiable with bounded, continuous first and second order derivatives. Then, for all $\alpha \in [0,\min\{1,\frac{3}{4}+\delta_0\})$ there exists $C_{\alpha,\gamma_0,\delta_0,p,T}(X_0,K_{\max}, \varphi)\in (0,\infty)$ such that
    \begin{align}\label{eq:weak X^Q_1 vs X^Q_2}
        \big|\mathbb{E}\big[\varphi(X^{Q_1}(T))\big] - \mathbb{E}\big[\varphi(X^{Q_2}(T))\big]\big| 
        \leq C_{\alpha,\gamma_0,\delta_0,p,T}(X_0,K_{\max}, \varphi)
    \| (-A)^{-\alpha} (Q_1-Q_2) \|_{\calL_1(L^2)},
    \end{align}
and, for all $M\in \N$:
    \begin{align*}
        \big|\mathbb{E}\big[\varphi(X^{Q_1}(T))\big] - \mathbb{E}\big[\varphi(X^{Q_2}_M(T))\big]\big| \leq C_{\alpha,\gamma_0,\delta_0,p,T}(X_0,K_{\max}, \varphi)
        \left( M^{-2\alpha} + 
    \| (-A)^{-\alpha} (Q_1-Q_2) \|_{\calL_1(L^2)} \right).
    \end{align*}
\end{corollary}

\begin{remark}[On bounding $\| (-A)^{-\alpha} (Q_1-Q_2) \|_{\calL_1(L^2)}$.] \label{rem:bounding Q}
First of all note that for all $\alpha,\beta \geq 0$ satisfying $0\leq \beta < \alpha -\frac{1}{2}$ we have $\| (-A)^{-\alpha+\beta} \|_{\calL_1(L^2)}<\infty$ by \eqref{eq: A L_1 reg} and (by~\eqref{eq:ideal})
\begin{equation}
\| (-A)^{-\alpha} (Q_1-Q_2) \|_{\calL_1(L^2)}
\leq \| (-A)^{-\alpha+\beta} \|_{\calL_1(L^2)}
\| (-A)^{-\beta}(Q_1-Q_2) \|_{\calL(L^2)}
\end{equation}
(note that the right-hand side involves a operator norm). Thus, under the assumptions of Theorem~\ref{thm:weak perturbation} there exists, for all $\beta \in [0,\min\{\frac{1}{2},\frac{1}{4}+\delta_0\})$, a constant $C_{\beta,\gamma_0,\delta_0,p,T}(X_0,K_{\max},\varphi) \in (0,\infty)$ such that
    \begin{align*}
        \sup_{M\in \N}\big|\mathbb{E}\big[\varphi(X^{Q_1}_M(T))\big] - \mathbb{E}\big[\varphi(X^{Q_2}_M(T))\big]\big|  
        \leq 
        C_{\beta,\gamma_0,\delta_0,p,T}(X_0,K_{\max},\varphi)
        \| (-A)^{-\beta} (Q_1-Q_2) \|_{\calL(L^2)}.
    \end{align*}
Also note that one can obtain better estimates if $Q_1$, $Q_2$, and $A$ are jointly diagonalisable (i.e., if the eigenvectors of $Q_1$ and $Q_2$ correspond to the eigenvectors $(h_k)_{k\in \N}$ of $A$), indeed, in this case we obtain
\begin{equation*}
\| (-A)^{-\alpha} (Q_1-Q_2) \|_{\calL_1(L^2)} = \sum_{k\in \N} (\pi k)^{-2\alpha} |\lambda_{Q_1}(h_k)- \lambda_{Q_2}(h_k)|,
\end{equation*}
where $\lambda_{Q_i}(h_k)$ is the eigenvalue of $Q_i$ ($i\in \{1,2\})$ corresponding to the eigenvector $h_k$.
\end{remark}

In particular, taking $Q_1 = \sum_{k=1}^{\infty}q_k \langle \cdot , e_k\rangle e_k$ and $Q_2 = \sum_{k=1}^{N} q_k \langle \cdot , e_k\rangle e_k$ in Corollary~\ref{cor:weak conv X X^{Q_2} X^{Q_2}_M} and using the rather crude estimate
\begin{equation*}
\| (-A)^{-\frac{3}{4}} (Q_1-Q_2) \|_{\calL_1(L^2)}
\leq C \left\| \sum_{k=N+1}^{\infty}q_k \langle \cdot , e_k\rangle e_k \right\|_{\calL(L^2)} = C q_{N+1}
\end{equation*}
we obtain:

\begin{corollary}\label{cor: weak conv X X^{Q_N} X^{Q_N}_M}
Assume there exist $\delta_0,\gamma_0 >0$ and $p>32$ such that $\EE{\exp(\gamma_0\Ltwo{X_0}^2)} < \infty$ and $\EE{\Ltwo{(-A)^{\frac{1}{4}+\delta_0}X_0}^p} < \infty$. Let $(q_k)_{k\in \N}\in \ell^1$ be a non-increasing sequence of non-negative real numbers and $(e_k)_{k\in \N}$ an orthonormal basis in $L^2$, and let  
\begin{equation}
Q=\sum_{k=1}^{\infty}q_k \langle \cdot , e_k\rangle e_k, \qquad 
Q_N=\sum_{k=1}^{N} q_k \langle \cdot , e_k\rangle e_k.
\end{equation}
Let $X^{Q}$, $X^{Q_N}$, and $X^{Q_N}_M$ ($N, M\in \N$) be solutions to~\eqref{eq:Burgers}  with $Q=Q$, to~\eqref{eq:Burgers} with $Q=Q_N$, and to~\eqref{eq:Galerkin approx} with $Q=Q_N$.
Then, there exists $C_{\gamma_0,\delta_0,p,T}(X_0,Q,\varphi)\in (0,\infty)$ such that
    \begin{align}\label{eq:weak X^Q vs X^Q_N}
        \big|\mathbb{E}\big[\varphi(X^{Q}(T))\big] - \mathbb{E}\big[\varphi(X^{Q_N}(T))\big]\big|
        \leq C_{\gamma_0,\delta_0,p,T}(X_0,Q, \varphi)\,
         q_{N+1},
    \end{align}
and for all $\alpha \in [0,\min\{1,\frac{3}{4}+\delta_0\})$ there exists $C_{\alpha,\gamma_0,\delta_0,p,T}(X_0,Q,\varphi)\in (0,\infty)$ such that for all $M\in \N$
    \begin{align}\label{eq: weak conv X X^{Q_N}_M 2}
        \big|\mathbb{E}\big[\varphi(X^{Q}(T))\big] - \mathbb{E}\big[\varphi(X^{Q_N}_M(T))\big]\big|  
        \leq C_{\alpha,\gamma_0,\delta_0,p,T}(X_0,Q, \varphi)
        \left( M^{-2\alpha}
        + q_{N+1} \right).
    \end{align}
\end{corollary}

\begin{proof}[Proof of Theorem~\ref{thm:weak perturbation}]
    Let $M\in \N$. Firstly, note that $\EE{\exp(\gamma_0\|X_0\|_{L^2}^2)} < \infty$ implies $\EE{\Ltwo{X_0}^4}<\infty$. Then, by the tower property for conditional expectation, the definition of $u^{Q_1}_M$, \cite[Lemma 5.3]{Brehier_Cox_Millet_2024}, and the fact that $X^{Q_1}_M(0) = P_MX_0 = X^{Q_2}_{M}(0)$, 
\begin{align*}
    \EE{\varphi(X^{Q_1}_M(T))} - \EE{\varphi(X^{Q_2}_M(T))}
    & \hspace{0.3em}= \EE{u^{Q_1}_M(T,X^{Q_2}_M(0))-u^{Q_1}_M(0,X^{Q_2}_M(T))} =\colon \text{err}.
\end{align*}
Applying Itô formula to the process $u^{Q_1}_M(T-t,X^{Q_2}_M(t))$ on $t\in [0,T]$, recalling the evolution equation \eqref{eq:Galerkin approx} (with $Q = Q_2$) for $X^{Q_2}_M$, and taking expectation, one obtains
\begin{align*}
    \text{err} = &\int_0^T \mathbb{E}\Big[\frac{\partial u^{Q_1}_M}{\partial t}(T-t,X^{Q_2}_M(t))\Big] dt \\
    &-\int_0^T \EE{Du^{Q_1}_M(T-t,X^{Q_2}_M(t)).(A X^{Q_2}_M(t) + B_M(X^{Q_2}_M(t)))}dt \\
    &-\frac{1}{2}\int_0^T \EE{\tr(D^2 u^{Q_1}_M(T-t,X^{Q_2}_M(t))P_M Q_2 P_M )}dt.
\end{align*}
Recalling that $u^{Q_1}_M$ satisfies the Kolmogorov equation \eqref{eq:Kolmogorov}, above becomes
\begin{align}
    \text{err} = \frac{1}{2}\int_0^T \EE{\tr(D^2 u^{Q_1}_M(T-t,X^{Q_2}_M(t))P_M (Q_1-Q_2) P_M )}dt. \label{eq: weak proof 1}
\end{align}
Now, by Lemma~\ref{lem:D^2u bound} (see also Remark~\ref{rem:dependence in Q}), for $\delta, \epsilon>0$ (to be specified later) there exists $C_{\alpha,\delta,\epsilon,T}(\varphi) \in (0,\infty)$ and non-decreasing $F_{p,T}:[0,\infty)\to (0,\infty)$ such that 
\begin{align*}
 &\tr(D^2 u^{Q_1}_M(T-t,X^{Q_2}_M(t))P_M(Q_1-Q_2)P_M )\\
 &\leq \| D^2 u^{Q_1}_M(T-t,X^{Q_2}_M(t))P_M(Q_1-Q_2)P_M  \|_{\calL_1(L^2)}
 \\
    & \leq \| D^2 u^{Q_1}_M(T-t,X^{Q_2}_M(t)) (-A)^{\alpha} \|_{\calL(H_M)}  \| (-A)^{-\alpha} (Q_1-Q_2) \|_{\calL_1(L^2)}
    \\
    &\leq C_{\alpha,\delta,\epsilon,T}(\varphi) F_{p,T}(\tr(Q_1))(T-t)^{-\alpha} e^{\epsilon\|X^{Q_2}_M(t)\|^2_{L^2}}
    \left(1+\|(-A)^{\frac{1}{4}+\delta}X^{Q_2}_M(t)\|^{16}_{L^2}\right)
    \| (-A)^{-\alpha} (Q_1-Q_2) \|_{\calL_1(L^2)}.
\end{align*}
Inserting this in \eqref{eq: weak proof 1}, one obtains
\begin{align*}
    |\text{err}| &\leq  C_{\alpha,\delta,\epsilon,T}(\varphi) F_{p,T}(K_{\max})\int_0^T (T-t)^{-\alpha} \EE{\Psi_{\delta,\epsilon,16}(X^{Q_2}_M(t))} dt \cdot \| (-A)^{-\alpha} (Q_1-Q_2) \|_{\calL_1(L^2)}\\
    &\leq C_{\alpha,\delta,\epsilon,T}(\varphi) F_{p,T}(K_{\max}) 
    \tfrac{T^{1-\alpha}}{1-\alpha} \,\EE{\sup_{t\in [0,T]}\Psi_{\delta,\epsilon,16}(X^{Q_2}_M(t))}
    \| (-A)^{-\alpha} (Q_1-Q_2) \|_{\calL_1(L^2)}.
\end{align*}
Lastly, choose $\delta = \frac{\delta_0}{2}, \epsilon = \frac{p-32}{2p}\frac{\gamma_0}{1+2\gamma_0K_{\max}}$. Note that since $\|Q_2\|_{\mathcal{L}(L^2)} \leq \tr(Q_2) \leq K_{\max}$, one has $\epsilon \in (0,\frac{p-32}{p}\frac{\gamma_0}{1+2\gamma_0\|Q_2\|_{\mathcal{L}(L^2)}})$. Hence, by Lemma~\ref{lem:sup of psi} (with $Q=Q_2,q=16$) and noting that the constant can be made independent of $Q_2$ by using $\tr(Q_2) \leq K_{\max}$, we conclude the result.
\end{proof}

\section{Strong convergence}\label{sec: strong bound}

The main aim of this section is to establish bounds for the strong error after the perturbation of the noise, i.e., to provide bounds for
\begin{align*}
    \sup_{t\in [0,T]}\|X^{Q_1}(t) - X^{Q_2}(t)\|_{L^r(\Omega; L^2)}, \quad r\in [1,\infty).
\end{align*}
where $X^{Q_1}$ and $X^{Q_2}$ are the solutions to~\eqref{eq:Burgers} with $Q=Q_1$ and $Q=Q_2$, $Q_1, Q_2\in \calL_1(L^2)$ are positive self-adjoint; see Corollary~\ref{cor:strong conv X X^{Q_2} X^{Q_2}_M} below. Once again, these bounds are obtained by first proving analogous, dimension-independent bounds for the Galerkin approximations, see Theorem~\ref{thm:strong perturbation} below. By combining Theorem~\ref{thm:strong perturbation} with the strong convergence rates for spectral Galerkin approximations of the stochastic Burgers equation obtained in~\cite{Hutzenthaler_Jentzen_2020}, we obtain the desired bounds. Again, this approach also allows us to establish strong convergence rates for an approximation obtained by applying a spectral Galerkin approximation in the spatial parameter and considering a Karhunen-Lo\'eve approximation of the noise; see Corollary~\ref{cor: strong conv X X^{Q_N} X^{Q_N}_M}.

\begin{theorem} \label{thm:strong perturbation}
    Suppose that there exist $\gamma_0,\delta_0 >0$ and $p>2 $ such that $\EE{\exp(\gamma_0 \|X_0\|_{L^2}^2)} < \infty$ and $\EE{\|(-A)^{\frac{1}{4}+\delta_0}X_0\|_{L^2}^{2p}} < \infty$. Let $K_{\max} \in (0,\infty)$ and let $Q_1,Q_2 \in \mathcal{L}_1(L^2)$ be positive and self-adjoint, such that $\max\{\tr(Q_1),\tr(Q_2)\} \leq K_{\max}$ and let $X^{Q_1}_M$, $X^{Q_2}_M$ ($M\in \N$) be solutions to~\eqref{eq:Galerkin approx} with $Q=Q_1$ and $Q=Q_2$. Then, for all $r\in [1,p)$, $\alpha\in [0,1)$ there exists $C_{\alpha,\gamma_0,\delta_0,p,r,T}(X_0,K_{\max}) \in (0,\infty)$ such that 
    \begin{align*}
        \sup_{M\in \N} \sup_{t\in [0,T]}\|X^{Q_1}_M(t) - X^{Q_2}_M(t)\|_{L^r(\Omega; L^2)} \leq C_{\alpha, \gamma_0,\delta_0,p,r,T}(X_0,K_{\max})
        \big\| (-A)^{-\nicefrac{\alpha}{2}} \big|Q_1^{\nicefrac{1}{2}} -Q_2^{\nicefrac{1}{2}}\big| \big\|_{\calL_2(L^2)}.
    \end{align*}
\end{theorem}

The proof of this theorem is postponed to the end of this section. In order to obtain bounds when $M\rightarrow \infty$ we recall the following, see  \cite[Equation 107]{Hutzenthaler_Jentzen_2020}, also \cite[Remark 5.2]{Brehier_Cox_Millet_2024}: 

\begin{theorem}\label{thm: strong conv Galerkin} Assume that there exist $\gamma_0,\delta_0>0$ and $p>32$ such that $\EE{\exp(\gamma_0\|X_0\|_{L^2}^2)} < \infty$ and $\EE{\|(-A)^{\frac{1}{4}+\delta_0}X_0\|_{L^2}^p} < \infty$. Let $Q\in \mathcal{L}_1(L^2)$ be positive and self-adjoint and let $X^{Q}$, $X^{Q}_M$ ($M\in \N$) be solutions to~\eqref{eq:Burgers} and~\eqref{eq:Galerkin approx}. Then, for all $r\in [1,\frac{p}{4})$, $\alpha \in (0,\min\{1,\frac{1}{2}+2\delta_0\})$ there exists $C_{\alpha,\gamma_0,\delta_0,p,r,T}(X_0)\in (0,\infty)$ and an increasing function $F_{p,T}\colon [0,\infty)\rightarrow [0,\infty)$ such that for all $M\in \N$
\begin{align}
    \sup_{t \in [0,T]}\|X^Q(t) - X^Q_M(t)\|_{L^r(\Omega;L^2)} \leq C_{\alpha,\gamma_0,\delta_0,p,r,T}(X_0) F_{p,T}(\tr(Q)) M^{-\alpha}. \label{eq:strong conv of X_M}
\end{align}
\end{theorem}

Combining Theorems~\ref{thm:strong perturbation} and~\ref{thm: strong conv Galerkin} we obtain:
\begin{corollary}\label{cor:strong conv X X^{Q_2} X^{Q_2}_M}
Assume that there exist $\gamma_0,\delta_0>0$ and $p>32$ such that $\EE{\exp(\gamma_0\|X_0\|_{L^2}^2)} < \infty$ and $\EE{\|(-A)^{\frac{1}{4}+\delta_0}X_0\|_{L^2}^p} < \infty$. Let $K_{\max} \in (0,\infty)$ and let $Q_1,Q_2 \in \mathcal{L}_1(L^2)$ be positive and self-adjoint, such that $\max\{\tr(Q_1),\tr(Q_2)\} \leq K_{\max}$ and let $X^{Q_1}$, $X^{Q_2}$, and $X^{Q_2}_M$ ($M\in \N$) be solutions to~\eqref{eq:Burgers} with $Q=Q_1$, to~\eqref{eq:Burgers} with $Q=Q_2$, and to~\eqref{eq:Galerkin approx} with $Q=Q_2$. Then, for all $r\in [1,\frac{p}{4})$, $\alpha \in (0,\min\{1,\frac{1}{2}+2\delta_0\})$ there exists  $C_{\alpha,\gamma_0,\delta_0,p,r,T}(X_0,K_{\max}) \in (0,\infty)$ such that 
    \begin{align}\label{eq:strong X^Q_1 vs X^Q_2}
        \sup_{t\in [0,T]}\|X^{Q_1}(t) - X^{Q_2}(t)\|_{L^r(\Omega; L^2)} \leq C_{\alpha, \gamma_0,\delta_0,p,r,T}(X_0,K_{\max})
        \big\| (-A)^{-\nicefrac{\alpha}{2}} \big|Q_1^{\nicefrac{1}{2}} -Q_2^{\nicefrac{1}{2}}\big| \big\|_{\calL_2(L^2)},
    \end{align}
and, for all $M\in \N$:
    \begin{align*}        
        \sup_{t\in [0,T]}\|X^{Q_1}(t) - X^{Q_2}_M(t)\|_{L^r(\Omega; L^2)}  \leq
        C_{\alpha,\gamma_0,\delta_0,p,r,T}(X_0,K_{\max})\Big( M^{-\alpha}
        +
        \big\| (-A)^{-\nicefrac{\alpha}{2}} \big|Q_1^{\nicefrac{1}{2}} -Q_2^{\nicefrac{1}{2}}\big| \big\|_{\calL_2(L^2)}\Big).
    \end{align*}
\end{corollary}

\begin{remark}[On bounding $\left\| (-A)^{-\nicefrac{\alpha}{2}} \left|Q_1^{\nicefrac{1}{2}} -Q_2^{\nicefrac{1}{2}}\right| \right\|_{\calL_2(L^2)}$]\label{rem:bounding Q HS}
See also Remark~\ref{rem:bounding Q}: for $\alpha,\beta \geq 0$ satisfying $0\leq \beta < \alpha -\frac{1}{2}$ we have $\| (-A)^{-\nicefrac{(\alpha-\beta)}{2}} \|_{\calL_2(L^2)}<\infty$ by \eqref{eq: A L_1 reg} and (by~\eqref{eq:ideal})
\begin{equation*}
\big\| (-A)^{-\nicefrac{\alpha}{2}} \big|Q_1^{\nicefrac{1}{2}} -Q_2^{\nicefrac{1}{2}}\big| \big\|_{\calL_2(L^2)} \leq \| (-A)^{-\nicefrac{(\alpha-\beta)}{2}} \|_{\calL_2(L^2)} \big\| (-A)^{-\nicefrac{\beta}{2}} (Q_1^{\nicefrac{1}{2}} -Q_2^{\nicefrac{1}{2}}) \big\|_{\calL(L^2)}.
\end{equation*}
Moreover, once again one can obtain better estimates if $Q_1$, $Q_2$, and $A$ are jointly diagonalisable, indeed, in this case we obtain
\begin{equation*}
\big\| (-A)^{-\nicefrac{\alpha}{2}} \big|Q_1^{\nicefrac{1}{2}} -Q_2^{\nicefrac{1}{2}}\big| \big\|_{\calL_2(L^2)} = \left(\sum_{k\in \N} (\pi k)^{-2\alpha} \big|\lambda_{Q_1}^{\nicefrac{1}{2}}(h_k)- \lambda_{Q_2}^{\nicefrac{1}{2}}(h_k)\big|^2\right)^{\nicefrac{1}{2}},
\end{equation*}
where $\lambda_{Q_i}(h_k)$ is the eigenvalue of $Q_i$ ($i\in \{1,2\})$ corresponding to the eigenvector $h_k$.
\end{remark}

In particular, taking $Q_1 = \sum_{k=1}^{\infty}q_k \langle \cdot , e_k\rangle e_k$ and $Q_2 = \sum_{k=1}^{N} q_k \langle \cdot , e_k\rangle e_k$ in Corollary~\ref{cor:strong conv X X^{Q_2} X^{Q_2}_M} and using the rather crude estimate
\begin{equation*}
\big\| (-A)^{-\frac{1}{2}-\min\{\frac{1}{4},\delta_0\}} \big|Q_1^{\nicefrac{1}{2}}-Q_2^{\nicefrac{1}{2}}\big| \big\|_{\calL_2(L^2)}
\leq C \left\| \sum_{k=N+1}^{\infty}\sqrt{q_k} \langle \cdot , e_k\rangle e_k \right\|_{\calL(L^2)} = C \sqrt{q_{N+1}}
\end{equation*}
we obtain:
\begin{corollary} \label{cor: strong conv X X^{Q_N} X^{Q_N}_M}
    Suppose that there exist $\gamma_0,\delta_0 >0,$ $p>32 $ such that $\EE{\exp(\gamma_0 \|X_0\|_{L^2}^2)} < \infty$ and $\EE{\|(-A)^{\frac{1}{4}+\delta_0}X_0\|_{L^2}^{p}} < \infty$. Let $(q_k)_{k\in \N}\in \ell^1$ be a non-increasing sequence of non-negative real numbers and $(e_k)_{k\in \N}$ an orthonormal basis in $L^2$, and let  
\begin{equation}
Q=\sum_{k=1}^{\infty}q_k \langle \cdot , e_k\rangle e_k, \qquad 
Q_N=\sum_{k=1}^{N} q_k \langle \cdot , e_k\rangle e_k.
\end{equation}
Let $X^{Q}$, $X^{Q_N}$, and $X^{Q_N}_M$ ($N, M\in \N$) be solutions to~\eqref{eq:Burgers}  with $Q=Q$, to~\eqref{eq:Burgers} with $Q=Q_N$, and to~\eqref{eq:Galerkin approx} with $Q=Q_N$. Then, for all $r\in [1,\frac{p}{4})$ there exists $C_{\gamma_0,\delta_0,p,r,T}(X_0,Q) \in (0,\infty)$ such that for all $N\in \N$,
    \begin{align}\label{eq:strong X^Q vs X^Q_N}
        \sup_{t\in [0,T]}\|X(t) - X^{Q}(t)\|_{L^r(\Omega; L^2)} \leq C_{\gamma_0,\delta_0,p,r,T}(X_0,Q)\sqrt{q_{N+1}}
    \end{align}
and for all $r\in [1,\frac{p}{4})$, $\alpha \in (0,\min\{1,\frac{1}{2}+2\delta_0\})$ there exists $C_{\alpha,\gamma_0,\delta_0,p,r,T}(X_0,Q) > 0$ such that for all $M\in \N$
    \begin{align}\label{eq:strong X^Qvvs X_M^Q_N}
        \sup_{t\in [0,T]}\|X^{Q}(t) - X^{Q_N}_M(t)\|_{L^r(\Omega; L^2)} \leq C_{\alpha, \gamma_0,\delta_0,p,r,T}(X_0,Q)
        \left( M^{-\alpha} + \sqrt{q_{N+1}}\right).
    \end{align}
\end{corollary}

\begin{remark}\label{rem:sharp strong bound}
We suspect that the bounds obtained in Theorem~\ref{thm: strong conv Galerkin} are essentially sharp: indeed, assume $Q_1$, $Q_2$ and $A$ are jointly diagonalisable and consider the Ornstein-Uhlenbeck processes $Y^{Q_i}$, $i\in \{1,2\}$ given by 
\begin{align*}
Y^{Q_i}(t) & = \int_0^{t} e^{(t-s)A} \,dW^{Q_i}(s) = \sum_{k=1}^{\infty} \lambda^{\nicefrac{1}{2}}_{Q_i}(h_k) h_k\int_0^{t} e^{-(t-s)k^2\pi^2}\,dW^{(k)}(s), \quad i\in \{1,2\}.
\end{align*}
It follows from It\^o's isometry that
\begin{equation*}
\E \big[\big\| Y^{Q_1}(1) - Y^{Q_2}(1) \big\|_{L^2}^2 \big] \simeq \big\| (-A)^{-\nicefrac{1}{2}}\big|Q_1^{\nicefrac{1}{2}}-Q_2^{\nicefrac{1}{2}}\big| \big\|^2_{\mathcal{L}_2(L^2)}.
\end{equation*}
\end{remark}
\begin{remark}
The conditions under which~\eqref{eq:strong X^Q_1 vs X^Q_2} and ~\eqref{eq:strong X^Q vs X^Q_N} (respectively,~\eqref{eq:weak X^Q_1 vs X^Q_2} and~\eqref{eq:weak X^Q vs X^Q_N}) hold can presumably be weakened to the assumptions in Theorem \ref{thm:strong perturbation} (respectively, Theorem~\ref{thm:weak perturbation}) by using Fatou's lemma and e.g.\ \cite[Corollary 4.5]{CoxEtAl:2021} instead of Theorem~\ref{thm: strong conv Galerkin} (respectively, Theorem~\ref{thm: weak rates Galerkin approx}).
\end{remark}

The proof of Theorem~\ref{thm:strong perturbation} is an adaptation of the proof for the strong convergence of $X^Q_M$ to $X^Q$, as in \cite[Section 3.2.3 and Proposition 3.7]{Hutzenthaler_Jentzen_2020}. In particular, the perturbation estimate in \cite[Corollary 2.11]{Hutzenthaler_Jentzen_2020} plays a key role. We state it here with $H = U= L^2, \epsilon = 0, F_1 = F_2 =: F$ and $B_i(x)=R_i^{\nicefrac{1}{2}}$ constant in $x$, where $R_i\in \mathcal{L}_1(L^2)$ are positive, self-adjoint for $i=1,2$:
\begin{theorem}[{\hspace{1sp}\cite[Corollary 2.11]{Hutzenthaler_Jentzen_2020}}] \label{thm:Cor 2.11}
    Let $A:D(A) \subseteq L^2 \to L^2$ be a densely defined linear operator, let $\mathcal{O}\subseteq D(A)$ and let $F \colon \mathcal{O} \rightarrow L^2$ be measurable. Furthermore, let $\chi \colon [0,T]\times \Omega \to \R$, and let $ \Hat{X}\colon [0,T]\times \Omega \to L^2$, $X_1,X_2 \colon [0,T]\times \Omega \to \mathcal{O}$ be predictable stochastic processes, which satisfy $\int_0^T \|AX_i(s)\|_{L^2} +\|F(X_i(s))\|_{L^2}\,ds< \infty $ a.s.\ for $i=1,2$ and $\int_0^T  \|A\Hat{X}(s)\|_{L^2}  + \|F(\Hat{X}(s))\|_{L^2} \,ds< \infty $ a.s., and moreover satisfy the stochastic evolution equations $X_i(t) = X_i(0) + \int_0^t AX_i(s) + F(X_i(s))\, ds +  W^{R_i}(t)$ for $i =1,2$ and $\Hat{X}(t) = X_2(0) + \int_0^t A\Hat{X}(t) + F(X_1(s)) \, ds + W^{R_2}(t)$. Assume that
\begin{align*}
\int_0^T \left[\frac{\langle X_2(t) - \Hat{X}(t), A(X_2(t) -\Hat{X}(t)) + F(X_2(t)) -F(\Hat{X}(t))\rangle_{L^2}}{\|X_2(t) -\Hat{X}(t)\|_{L^2} ^2} + \chi(t)\right]^+dt < \infty.
\end{align*}

Then, for all $t \in [0,T]$, $p \geq 2$ and $ q, r >0$ such that $\frac{1}{p} + \frac{1}{q} = \frac{1}{r}$, it holds that

\begin{align*}
    &\|X_1(t) - X_2(t)\|_{L^r(\Omega; L^2)} \leq \|X_1(t) - \Hat{X}(t)\|_{L^r(\Omega; L^2)} \\
    &\hspace{1em}+\Big\|p\|X_2 - \Hat{X}\|_{L^2} ^{p-2}\left[\langle X_2 - \Hat{X}, F(\Hat{X})- F(X_1)\rangle_{L^2}  - \chi \|X_2 - \Hat{X}\|_{L^2}^2 \right]^+ \Big\|^{\frac{1}{p}}_{L^1([0,t]\times \Omega; \R)} \\
    &\hspace{1em}\times \Bigg\|\exp\left(\int_0^t \left[\frac{\langle X_2(s) - \Hat{X}(s), A(X_2(s) -\Hat{X}(s)) + F(X_2(s)) -F(\Hat{X}(s))\rangle_{L^2}}{\|X_2(s) -\Hat{X}(s)\|_{L^2} ^2} + \chi(s)\right]^+\,ds\right) \Bigg\|_{L^q(\Omega;\R)}.
\end{align*}

\end{theorem}

We are now ready to prove Theorem~\ref{thm:strong perturbation}:
\begin{proof}[Proof of Theorem~\ref{thm:strong perturbation}]
Let $M \in \N$. Because of \eqref{eq:A_fracpownorminc}, without loss of generality we may assume $\delta_0 <\frac{1}{4}$.  We will use Theorem~\ref{thm:Cor 2.11} with $p,r$ as given; $q =  \left(\frac{1}{r} - \frac{1}{p}\right)^{-1}$, $\mathcal{O} = H_M$, $X_1 = X^{Q_1}_M, X_2 = X^{Q_2}_M$, so that $A$ is the Dirichlet Laplace operator, $X_1(0) = X_2(0) = P_M X_0, F = B_M, R_i = P_M Q_i P_M$, $i\in \{1,2\}$. Set $\Hat{X} = X^{Q_1}_M - Y_M$, where $Y_M$ satisfies the evolution equation
\begin{align}
    Y_M(t) = \int_0^t A Y_M(s)\, ds + P_M(W^{Q_1}(t) - W^{Q_2}(t)),\quad \forall t\in [0,T]. \label{eq:Y_M evo eq}
\end{align}
Note that $W^{Q_1}(t) - W^{Q_2}(t)$ is a $|Q_1^{\nicefrac{1}{2}}-Q_2^{\nicefrac{1}{2}}|^2$-Brownian motion. Therefore, the stochastic convolution $Y_M$ given by
\begin{align}
    Y_M(t) = \int_0^t e^{(t-s)A}P_M\,dW^{|Q_1^{\nicefrac{1}{2}}-Q_2^{\nicefrac{1}{2}}|^2}(s),\quad \forall t\in [0,T], \label{eq:Y_M rep 1}
\end{align}
satisfies \eqref{eq:Y_M evo eq}. Moreover, $\Hat{X}$ satisfies
\begin{align*}
\Hat{X}(t) &= X^{Q_1}_M(t) - Y_M(t) \\
&=P_MX_0 + \int_0^t A X^{Q_1}_M(s) + B_M(X^{Q_1}_M(s)) \, ds +  P_M W^{Q_1}(t) - \int_0^t A Y_M(s) \,ds  - P_M(W^{Q_1}(t) - W^{Q_2}(t)) \\
&=P_MX_0 + \int_0^t A(X^{Q_1}_M(s)-Y_M(s)) + B_M(X^{Q_1}_M(s)) \, ds + P_M W^{Q_2}(t)\\
&=X_2(0) + \int_0^t A \hat{X}(t) + F(X_1(s)) \,ds + W^{R_2}(t).
\end{align*}

To ease the notation, denote $Z_M = X^{Q_2}_M - X^{Q_1}_M +Y_M \in H_M$. Given a $\chi$ that is specified later, Theorem~\ref{thm:Cor 2.11} asserts that if
\begin{align}\int_0^T \left[\frac{\langle Z_M(t), AZ_M(t) + B_M(X^{Q_2}_M(t)) -B_M(X^{Q_1}_M(t) - Y_M(t))\rangle_{L^2}}{\|Z_M(t)\|_{L^2}^2} + \chi(t)\right]^+ dt < \infty, \label{eq:condition}
\end{align}
then
\begin{align}
    \|X^{Q_1}_M(t) - X^{Q_2}_M(t)\|_{L^r(\Omega; L^2)} \leq (A) +  (B)\times (C),
\end{align}
where 
\begin{align*}
    &(A) := \|Y_M(t)\|_{L^r(\Omega; L^2)}, \\
    &(B) := \Big\|p\|Z_M\|_{L^2}^{p-2}\left[\langle Z_M, B_M(X^{Q_1}_M - Y_M)- B_M(X^{Q_1}_M)\rangle_{L^2} - \chi \|Z_M\|_{L^2}^2 \right]^+ \Big \|^{\frac{1}{p}}_{L^1([0,t]\times \Omega; \R)} ,\\
    &(C) := \Bigg\|\exp\Big(\int_0^t \bigg[\frac{\langle Z_M(s), AZ_M(s) + B_M(X^{Q_2}_M(s)) -B_M(X^{Q_1}_M(s) - Y_M(s))\rangle_{L^2}}{\|Z_M(s)\|_{L^2}^2} + \chi(s)\bigg]^+ds\Big) \Bigg\|_{L^q(\Omega;\R)}.
\end{align*}

We will provide bounds for each of $(A)-(C)$. Note that by proving bounds for $(C)$, we obtain in particular that \eqref{eq:condition} holds.\\
\textit{Estimate for $(A)$}: since $Y_M$ is a stochastic convolution given by \eqref{eq:Y_M rep 1}, using the Proposition~\ref{prop:stoch conv moment bound} with $Q=|Q_1^{\nicefrac{1}{2}}-Q_2^{\nicefrac{1}{2}}|^2$, for all $\alpha\in [0,1)$ and $r\in [1,\infty)$ there exists $C_{\alpha,r,T}>0$ such that
\begin{align}
    \|Y_M(t)\|_{L^r(\Omega; L^2)} \leq C_{\alpha,r,T}\big\| (-A)^{-\nicefrac{\alpha}{2}} \big|Q_1^{\nicefrac{1}{2}}-Q_2^{\nicefrac{1}{2}}\big| \big\|_{\calL_2(L^2)}. \label{eq:(A) bound}
\end{align}
\textit{Estimate for $(B)$}: Let 
\begin{align}
    \chi(s) = \Tilde{\chi}(s) +\frac{1/2-1/p}{T} -\frac{\langle Z_M(s), AZ_M(s) + B_M(X^{Q_2}_M(s)) -B_M(X^{Q_1}_M(s) - Y_M(s))\rangle_{L^2}}{\|Z_M(s)\|_{L^2}^2}, \label{eq:chi}
\end{align}
where $\Tilde{\chi}$ is specified later. With this choice, $(B)$ takes form
\begin{align}
    (B) &= \bigg\|p\|Z_M\|_{L^2}^{p-2}\bigg[\langle Z_M, B_M(X^{Q_1}_M - Y_M)- B_M(X^{Q_1}_M)\rangle_{L^2} + \langle Z_M, AZ_M + B_M(X^{Q_2}_M) -B_M(X^{Q_1}_M - Y_M)\rangle_{L^2} \nonumber\\
    &\quad-\left(\Tilde{\chi}(s) +\frac{1/2-1/p}{T}\right)\|Z_M\|_{L^2}^2\bigg]^+\bigg\|^{\frac{1}{p}}_{L^1([0,t]\times \Omega; \R)}. \label{eq:(B) form}
\end{align}
Recalling the definitions of $B_M(\cdot)$, $Z_M$ and the identity~\eqref{eq:B_cancelations} we obtain:
\begin{align*}
    &\langle Z_M, B_M(X^{Q_1}_M - Y_M)- B_M(X^{Q_1}_M)\rangle_{L^2} + \langle Z_M, AZ_M + B_M(X^{Q_2}_M) -B_M(X^{Q_1}_M - Y_M)\rangle_{L^2} 
    \\&\qquad = -\langle \nabla Z_M, (X^{Q_1}_M - Y_M)^2- (X^{Q_1}_M)^2\rangle_{L^2} -\|\nabla Z_M\|_{L^2}^2 + \tfrac{1}{2}\langle Z_M^2, \nabla(X^{Q_1}_M - Y_M + X^{Q_2}_M)\rangle_{L^2}\\
    &\qquad \leq \tfrac{1}{2}\|\nabla Z_M\|_{L^2}^2 + \tfrac{1}{2}\|Y_M(Y_M-2X^{Q_1}_M)\|_{L^2}^2 -\|\nabla Z_M\|_{L^2}^2+ \tfrac{1}{2}\|Z_M^2\|_{L^2}\|\nabla(X^{Q_1}_M - Y_M + X^{Q_2}_M)\|_{L^2}\\
    &\qquad  \leq \tfrac{1}{2}\bigLtwo{Y_M(Y_M-2X^{Q_1}_M)}^2 -\tfrac{1}{2}\|\nabla Z_M\|_{L^2}^2+ \tfrac{1}{2}\|Z_M\|_{L^2}\|Z_M\|_{L^\infty} \bigLtwo{\nabla(X^{Q_1}_M - Y _M+ X^{Q_2}_M)}.
\end{align*}
Furthermore, Young's inequality in the form $ab \leq \frac{a^2}{32\delta} + 8\delta b^2$ applied to the last term shows, for all $\delta > 0$,
\begin{equation}\label{eq:est 2}
\begin{aligned}
    &\|Z_M, B_M(X^{Q_1}_M - Y_M)- B_M(X^{Q_1}_M)\rangle_{L^2} + \langle Z_M, AZ_M + B_M(X^{Q_2}_M) -B_M(X^{Q_1}_M - Y_M)\rangle_{L^2}  \\
    &\qquad \leq \tfrac{1}{2}\bigLtwo{Y_M(Y_M-2X^{Q_1}_M)}^2 -\tfrac{1}{2}\|\nabla Z_M\|_{L^2}^2+ \tfrac{1}{32\delta}\|Z_M\|^2_{L^\infty}
    + 2\delta \|Z_M\|_{L^2}^2\bigLtwo{\nabla(X^{Q_1}_M - Y_M + X^{Q_2}_M)}^2 \\
    &\qquad \leq \tfrac{1}{2}\bigLtwo{Y_M(Y_M-2X^{Q_1}_M)}^2 + \Big(\kappa(\delta)+ 2\delta\bigLtwo{\nabla(X^{Q_1}_M - Y_M + X^{Q_2}_M)}^2\Big)\|Z_M\|_{L^2}^2, 
\end{aligned}
\end{equation}
where in the last line we use the existence of a strictly decreasing function $\kappa \colon (0,\infty) \to (0,\infty)$ with the property that $\frac{1}{32\delta}\|x\|^2_{L^\infty} \leq \kappa(\delta)\Ltwo{x}^2 + \frac{1}{2}\Ltwo{\nabla x}^2$ (see \cite[Section 3.2.3, p. 34]{Hutzenthaler_Jentzen_2020}). Now, with $\delta$ to be specified later, choose 
\begin{align}
    \Tilde{\chi} = \kappa(\delta)+ 2\delta\bigLtwo{\nabla(X^{Q_1}_M - Y_M + X^{Q_2}_M)}^2. \label{eq:chi tilde}
\end{align}
 Combining \eqref{eq:(B) form}, \eqref{eq:est 2} and the choice of $\Tilde{\chi}$ gives
\begin{align}
    (B)\leq \Big\|p\|Z_M\|_{L^2}^{p-2}\left[\tfrac{1}{2}\bigLtwo{Y_M(Y_M-2X^{Q_1}_M)}^2 - \tfrac{1/2-1/p}{T}\|Z_M\|_{L^2}^2\right]^+ \Big\|^{\frac{1}{p}}_{L^1([0,t]\times \Omega; \R)}.\label{eq:(B) bound 2}
\end{align}
Moreover, by Young's inequality with conjugate exponents $\frac{p}{2}$ and $\frac{p-2}{p}$, it holds that 
\begin{equation}
\begin{aligned}
    &\Big\|p\|Z_M\|_{L^2}^{p-2}\left[\tfrac{1}{2}\bigLtwo{Y_M(Y_M-2X^{Q_1}_M)}^2 - \tfrac{1/2-1/p}{T}\|Z_M\|_{L^2}^2\right]^+\Big\|^{\tfrac{1}{p}}_{L^1([0,t]\times \Omega; \R)} \label{eq:est 2 1}\\
    &\qquad = \bigg(p\int_0^t \EE{\left[\tfrac{1}{2}\bigLtwo{Y_M(s)(Y_M(s)-2X^{Q_1}_M(s))}^2\|Z_M(s)\|_{L^2}^{p-2} - \tfrac{1/2-1/p}{T}\|Z_M(s)\|_{L^2}^p\right]^+}\, ds\bigg)^{\frac{1}{p}} \\
    &\qquad \leq \bigg(p\int_0^t \EE{\left[\tfrac{1}{p} T^{\frac{p-2}{2}} \bigLtwo{Y_M(s)(Y_M(s)-2X^{Q_1}_M(s))}^p+\tfrac{p-2}{2pT}\|Z_M(s)\|_{L^2}^{p} - \tfrac{p-2}{2pT}\|Z_M(s)\|_{L^2}^p\right]^+}\, ds\bigg)^{\frac{1}{p}} \\
    &\qquad 
    \leq T^{\frac{p-2}{2p}}\bigg(\int_0^T \EE{\bigLtwo{Y_M(s)(Y_M(s)-2X^{Q_1}_M(s))}^p}\, ds\bigg)^{\frac{1}{p}}.
\end{aligned}
\end{equation}
Furthermore, one has
\begin{align}
    \EE{\bigLtwo{Y_M(s)(Y_M(s)-2X^{Q_1}_M(s))}^p} &\leq \EE{\big\| Y_M(s)-2X^{Q_1}_M(s)\big\|^p_{L^\infty}\bigLtwo{Y_M(s)}^p} \nonumber \\
    &\leq \EE{\big\|Y_M(s)-2X^{Q_1}_M(s)\big\|^{2p}_{L^\infty}}^\frac{1}{2}\EE{\bigLtwo{Y_M(s)}^{2p}}^\frac{1}{2}\label{eq:est 2 2} \\
    &\leq 2^{p}\left(\EE{\|Y_M(s)\|^{2p}_{L^\infty}}+\EE{\big\| 2X^{Q_1}_M(s)\big\|^{2p}_{L^\infty}}\right)^\frac{1}{2}\EE{\bigLtwo{Y_M(s)}^{2p}}^\frac{1}{2}. \nonumber
\end{align}
Now, note that  $\EE{\|X_0\|_{L^2}^{6p}} < \infty$ since $\EE{\exp(\gamma_0 \|X_0\|_{L^2}^2)} < \infty$. Hence, by Lemma~\ref{lem:L_inf bound of X_M} there exists a non-decreasing  $F_{p,T}\colon [0,\infty) \to  (0,\infty)$ and $C_{\delta_0,p,T}(X_0)\in (0,\infty)$ such that, for all $s \in [0,T]$, 
\begin{align*}
    \EE{\|X^{Q_1}_M(s)\|^{2p}_{L^\infty}} &\leq C_{\delta_0,p,T}(X_0)F_{p,T}(\tr(Q_1)) \leq C_{\delta_0,p,T}(X_0)F_{p,T}(K_{\max}).
\end{align*}
Moreover, recalling that $Y_M$ may be written as in \eqref{eq:Y_M rep 1}, by Lemma~\ref{lem:Y Linf bound} there exists $C_{p,T}>0$ such that for all $s\in [0,T]$ it holds that 
\begin{align*}
    \EE{\|Y_M(s)\|^{2p}_{L^\infty}} &\leq C_{p,T}\tr (|Q_1^{\nicefrac{1}{2}}-Q_2^{\nicefrac{1}{2}}|^2)^p  \leq C_{p,T}(2\tr(Q_1)+2\tr(Q_2))^p
    \leq C_{p,T}(4K_{\max})^p.
\end{align*} 
Combining equations \eqref{eq:est 2 1} and \eqref{eq:est 2 2}, the estimates above, and Proposition~\ref{prop:stoch conv moment bound}, we obtain that there exists $C_{\alpha,\delta_0,p,T}(X_0,K_{\max}) \in (0,\infty)$ such that
\begin{equation}
\label{eq:(B) bound 3}
\begin{aligned}
    &\left\|
        p\|Z_M\|_{L^2}^{p-2}
        \left[
            \tfrac{1}{2}
                \bigLtwo{Y_M(Y_M-2X^{Q_1}_M)}^2 - \frac{1/2-1/p}{T}
                \|
                    Z_M
                \|_{L^2}^2
        \right]^+\right\|^{\frac{1}{p}}_{L^1([0,t]\times \Omega; \R)}
        \\ &\qquad \leq C_{\delta_0,p,T}(X_0,K_{\max})\left[\int_0^T \EE{\bigLtwo{Y_M(s)}^{2p}}^\frac{1}{2}ds\right]^{\frac{1}{p}} \\
    &\qquad \leq C_{\alpha, \delta_0,p,T}(X_0,K_{\max}) 
    \big\| (-A)^{-\nicefrac{\alpha}{2}} \big| Q_1^{\nicefrac{1}{2}} - Q_2^{\nicefrac{1}{2}} \big| \big\|_{\calL_2(L^2)}.
\end{aligned}
\end{equation}
Hence, combining \eqref{eq:(B) bound 2} and \eqref{eq:(B) bound 3} shows
\begin{align}
     (B)\leq C_{\alpha, \delta_0,p,T}(X_0,K_{\max})\big\| (-A)^{-\nicefrac{\alpha}{2}} \big| Q_1^{\nicefrac{1}{2}} - Q_2^{\nicefrac{1}{2}} \big| \big\|_{\calL_2(L^2)}. \label{eq:(B) term estimate}
\end{align}
\textit{Estimate for $(C)$}: For the choice of $\chi$ and $\Tilde{\chi}$ in \eqref{eq:chi} and \eqref{eq:chi tilde}, one has 
\begin{align*}
    &\left[\frac{\langle Z_M(s), AZ_M(s) + B_M(X^{Q_2}_M(s)) -B_M(X^{Q_1}_M(s) - Y_M(s))\rangle_{L^2}}{\|Z_M(s)\|_{L^2}^2} + \chi(s)\right]^+ \\
    &\qquad = \frac{1/2-1/p}{T} + \kappa(\delta)+ 2\delta\bigLtwo{\nabla(X^{Q_1}_M(s) - Y_M(s) + X^{Q_2}_M(s))}^2 \\
    &\qquad \leq \frac{1/2-1/p}{T} + \kappa(\delta)+ 6\delta\bigLtwo{\nabla X^{Q_1}_M(s)}^2 + 6\delta \bigLtwo{\nabla Y_M(s)}^2 + 6\delta \bigLtwo{\nabla X^{Q_2}_M(s)}^2.
\end{align*}
Hence, by Hölder's inequality
\begin{align}
    (C)&=\Bigg\|\exp\Big(\int_0^t \Big[\frac{\langle Z_M(s), AZ_M(s) + B_M(X^{Q_2}_M(s)) -B_M(X^{Q_1}_M(s) - Y_M(s))\rangle_{L^2}}{\|Z_M(s)\|_{L^2}^2} + \chi(s)\Big]^+ds\Big)\Bigg\|_{L^q(\Omega;\R)} \nonumber\\
    &\leq e^{\frac{(1/2-1/p)t}{T} + t\kappa(\delta)} \bigg\|\exp\left(6 \delta \int_0^t \bigLtwo{\nabla X^{Q_1}_M(s)}^2 + \bigLtwo{\nabla X^{Q_2}_M(s)}^2 +\Ltwo{\nabla Y_M(s)}^2 ds\right)\bigg\|_{L^q(\Omega;\R)} \label{eq:(C) bound 1}\\
    &\leq e^{1/2-1/p + T\kappa(\delta)}\EE{\exp\left(\int_0^T 18q\delta\bigLtwo{\nabla X^{Q_1}_M(s)}^2 ds\right)}^{\frac{1}{3q}} \nonumber\\
    &\hspace{1em}\times \EE{\exp\left(\int_0^T 18q\delta\bigLtwo{\nabla X^{Q_2}_M(s)}^2 ds\right)}^{\frac{1}{3q}}\EE{\exp\left(\int_0^T 18q\delta\Ltwo{\nabla Y_M(s)}^2 ds\right)}^{\frac{1}{3q}}, \nonumber
\end{align}
where in the last inequality we also take $t=T$ as the expression was increasing in $t$.\\
Now, take $\delta = \frac{\gamma_0}{20q(1+8\gamma_0 K_{\max})}$. Then in particular $18q\delta \in (0,\frac{\gamma_0}{1+2\gamma_0 \|Q_1\|_{\mathcal{L}(L^2)}})$, so by \eqref{eq:X_M exp bound random initial} in Lemma~\ref{lem: X_M sup +exp moment bounds} and the fact that $\tr(Q_1)\leq K_{\max}$, one has 
\begin{equation}
\label{eq:(C) bound 2}
\begin{aligned}
    \EE{\exp\left(\int_0^T 18q\delta\bigLtwo{\nabla X^{Q_1}_M(s)}^2 ds\right)}  &\leq 2e^{18q\delta T  K_{\max}}\EE{\exp(\gamma_0\Ltwo{X_0}^2)}^{\frac{18q\delta}{\gamma_0}}.
\end{aligned}
\end{equation}
Similarly, 
\begin{equation}
\label{eq:(C) bound 3}
\begin{aligned}
    \EE{\exp\left(\int_0^T 18q\delta\bigLtwo{\nabla X^{Q_2}_M(s)}^2 ds\right)} &\leq 2e^{18q\delta T  K_{\max}}\EE{\exp(\gamma_0\Ltwo{X_0}^2)}^{\frac{18q\delta}{\gamma_0}}.
\end{aligned} 
\end{equation}
 In addition, since $18q\delta < \frac{\gamma_0}{1+8\gamma_0 K_{\max}} \leq \frac{\gamma_0}{1+2\gamma_0\||Q_1^{\nicefrac{1}{2}}-Q_2^{\nicefrac{1}{2}}|^2\|_{\mathcal{L}(L^2)}} < \frac{1}{2\||Q_1^{\nicefrac{1}{2}}-Q_2^{\nicefrac{1}{2}}|^2\|_{\mathcal{L}(L^2)}}$, by the exponential bound \eqref{eq:Y exp bound}
 \begin{align}
     \EE{\exp\left(\int_0^T 18q\delta\Ltwo{\nabla Y_M(s)}^2 ds\right)}\leq 2e^{18q\delta T \tr(|Q_1^{\nicefrac{1}{2}}-Q_2^{\nicefrac{1}{2}}|^2)} \leq 2e^{72 q\delta T K_{\max}}.\label{eq:(C) Y bound}
 \end{align}
 Thus, the estimates \eqref{eq:(C) bound 1}, \eqref{eq:(C) bound 2}, \eqref{eq:(C) bound 3}, \eqref{eq:(C) Y bound} show an existence of  $C_{\gamma_0,p,q,T}(X_0,K_{\max}) \in (0,\infty)$ such that
\begin{align}
    (C) \leq C_{\gamma_0,p,q,T}(X_0,K_{\max}). \label{eq:(C) term estimate} 
\end{align}\par 
Finally, by Theorem~\ref{thm:Cor 2.11} together with \eqref{eq:(A) bound}, \eqref{eq:(B) term estimate} and \eqref{eq:(C) term estimate}, there exists $C_{\alpha,\gamma_0,\delta_0,p,r,T}(X_0,K_{\max}) >0$ such that for all $t\in [0,T]$, it holds that
\begin{align*}
    \|X^{Q_1}_M(t) - X^{Q_2}_M(t)\|_{L^r(\Omega; L^2)} \leq C_{\alpha,\gamma_0,\delta_0,p,r,T}(X_0,K_{\max})
    \left\| (-A)^{-\nicefrac{\alpha}{2}} \left| Q_1^{\nicefrac{1}{2}} - Q_2^{\nicefrac{1}{2}} \right| \right\|_{\calL_2(L^2)}.
\end{align*}
\end{proof}

\appendix

\section{Proofs of Lemmas~\ref{lem:Y exp bound}--\ref{lem:sup of psi}}\label{App A}
All proofs presented here are minor adaptations of proofs in~\cite{Brehier_Cox_Millet_2024}. 
\begin{proof}[Proof of Lemma~\ref{lem:Y exp bound}]
This proof is an adaptation of the proof of Lemma 3.1 given in \cite[Appendix A]{Brehier_Cox_Millet_2024}. Recall the evolution equation \eqref{eq:Y_M evolution eq} for $Y_M$. Then applying the energy equality for $Y_M$ gives
    \begin{align*}
        \Ltwo{Y_M(t)}^2 = \int_0^t 2\innt{Y_M(s),AY_M(s)}\, ds + t\|P_MQ^{\nicefrac{1}{2}}\|^2_{\mathcal{L}_2(L^2)} + 2 \int_0^t \langle Y_M(s),P_M\,dW^{Q}(s)\rangle_{L^2}.
    \end{align*}
    Using integration by parts and multiplying by $\alpha>0$, one has
    \begin{align} \label{eq:Y exp proof 2}
        \alpha \Ltwo{Y_M(t)}^2 + \alpha &\int_0^t \Ltwo{\nabla Y_M(s)}^2\, ds= \alpha t\|P_MQ^{\nicefrac{1}{2}}\|^2_{\mathcal{L}_2(L^2)}\\
        &+ 2 \alpha \int_0^t \langle Y_M(s),P_M\,dW^{Q}(s)\rangle_{L^2} - \alpha \int_0^t \Ltwo{\nabla Y_M(s)}^2\, ds.
    \end{align}
    Set 
    \begin{align*}
        N(t) := \int_0^t \langle Y_M(s),P_M\,dW^{Q}(s)\rangle_{L^2}, \quad t \in [0,T],
    \end{align*}
    and 
    \begin{align*}
        Z_\alpha := \exp\left(\sup_{t\in [0,T]} \left( 2 \alpha \int_0^t \langle Y_M(s),P_M\,dW^{Q}(s)\rangle_{L^2} - \alpha \int_0^t \Ltwo{\nabla Y_M(s)}^2\, ds\right)\right).
    \end{align*}
    Note that $N(t)$ is a continuous, square integrable real-valued martingale. By \cite[Theorem 2.3]{Gawarecki_Mandrekar_2011}, its quadratic variation is given by 
    \begin{align*}
        \langle N \rangle_t = \int_0^t \|h \mapsto \langle Y_M(s),P_MQ^{\nicefrac{1}{2}}h\rangle_{L^2}\|^2_{\mathcal{L}_2(L^2,\,\R)} \, ds.
    \end{align*}
    Let $(f_k)_{k\in \N}$ be an arbitrary basis for $L^2$. Using Poincaré's inequality \eqref{eq:poincare} and $\|Q^{\nicefrac{1}{2}}\|^2_{\mathcal{L}(L^2)}=\|Q\|_{\mathcal{L}(L^2)}$,
    \begin{equation}
        \label{eq:exp Y proof 1}
    \begin{aligned}
        \langle N \rangle_t &= \int_0^t \sum_{k=1}^\infty |\langle Y_M(s),P_MQ^{\nicefrac{1}{2}}e_k\rangle_{L^2}|^2 \, ds  = \int_0^t \sum_{k=1}^\infty |\langle Q^{\nicefrac{1}{2}}Y_M(s),e_k\rangle_{L^2}|^2 \, ds  \\
        &= \int_0^t \|Q^{\nicefrac{1}{2}}Y_M(s)\|^2_{L^2} \, ds  \leq \frac{1}{2}\|Q\|_{\mathcal{L}(L^2)}\int_0^t \|\nabla Y_M(s)\|^2_{L^2} \, ds. 
    \end{aligned}
    \end{equation}
    Let $c(\alpha) = \frac{1}{\alpha\|Q\|_{\mathcal{L}(L^2)}}$, $\hat{N}(t) := \frac{2}{\|Q\|_{\mathcal{L}(L^2)}}N(t)$ and 
    \begin{align*}
        M_\alpha(t) := \exp\left(\hat{N}(t) - \frac{1}{2}\langle \hat{N} \rangle_t\right), \quad t \in [0,T].
    \end{align*}
    Note that $\langle \hat{N} \rangle_t = \frac{4}{\|Q\|^2_{\mathcal{L}(L^2)}} \langle N \rangle_t$. Then, since the estimate \eqref{eq:exp Y proof 1} holds, for all $K \in (0,\infty)$ one has
    \begin{align*}
        \Prob(Z_\alpha > e^K) &= \Prob(Z^{c(\alpha)}_\alpha > e^{c(\alpha)K})\\
        &\leq \Prob \left(\exp\bigg[\sup_{t\in[0,T]}\left(\frac{2}{\|Q\|_{\mathcal{L}(L^2)}}N(t) - \frac{2}{\|Q\|^2_{\mathcal{L}(L^2)}} \langle N \rangle_t\right)\bigg] >e^{c(\alpha)K}\right) \\
        &= \Prob\left( \exp\bigg[\sup_{t\in[0,T]}\left(\hat{N}(t) - \frac{1}{2} \langle \hat{N} \rangle_t \right)\bigg] > e^{c(\alpha)K} \right)\\
        &= \Prob\left(\sup_{t\in [0,T]}M_\alpha(t) > e^{c(\alpha)K}\right).
    \end{align*}
    Note that, by \cite[Proposition 3.4, p. 140]{Revuz_Yor_1991}, $M_\alpha(t)$ is a (positive) continuous local martingale and therefore a continuous supermartingale by \cite[ Exercise 1.46, p. 129]{Revuz_Yor_1991}. Hence, the maximal inequality in \cite[Exercise 1.15, p. 55]{Revuz_Yor_1991} implies
    \begin{align*}
        \Prob\left(\sup_{t\in [0,T]}M_\alpha(t) > e^{c(\alpha)K}\right) \leq e^{-c(\alpha)K}\EE{M_\alpha(0)}= e^{-c(\alpha)K}.
    \end{align*}
    Now assume $\alpha <\frac{2}{\|Q\|_{\mathcal{L}(L^2)}}$, so that $c(\alpha) > 2$. Combining the estimates above yields
    \begin{align*}
        \Prob(Z_\alpha > e^K) \leq e^{-2K}.
    \end{align*}
    In turn, change of variables shows
    \begin{align*}
        \EE{Z_\alpha} &= \int_0^\infty \Prob(Z_\alpha > z)\,dz \leq 1 +\int_1^\infty \Prob(Z_\alpha > z)\,dz = 1 + \int_0^\infty \Prob(Z_\alpha > e^{K})e^K\,dK \\
        &\leq 1 + \int_0^\infty e^{-K}\,dK = 2.
    \end{align*}
    Finally, taking supremum over $t\in [0,T]$ of \eqref{eq:Y exp proof 2}, exponentiating and taking expectation, we obtain
    \begin{align*}
        \EE{\exp\left(\alpha \sup_{t\in [0,T]}\Ltwo{Y_M(t)}^2+\alpha\int_0^T \Ltwo{\nabla Y_M(s)}^2 ds\right)} &\leq e^{\alpha T \|P_MQ^{\nicefrac{1}{2}}\|^2_{\mathcal{L}_2(L^2)}}\EE{Z_\alpha}\\
        &\leq 2 e^{\alpha T \tr(Q)}.
    \end{align*}
\end{proof}

\begin{remark}
    It should be possible to prove a slightly different version of Lemma~\ref{lem:Y exp bound} by using \cite[Corollary 2.4]{Cox_Hutzenthaler_Jentzen_2024}.
\end{remark}

\begin{proof}[Proof of Lemma~\ref{lem: X_M sup +exp moment bounds}]
    Note that the results \eqref{eq:X_M exp bound random initial} and \eqref{eq:X_M exp bound deterministic initial} are already specified with explicit dependence in \cite{Brehier_Cox_Millet_2024}, so it only remains to establish \eqref{eq:X_M p moment bound}. \\
    Define a stopping time $\tau_R = \inf\big( \big\{ t\in [0,T] \colon \big\| X_M^Q(t) \big\|_{L^2} \geq R \big\}\cup \{T\}\big)$ for all $R\in (0,\infty)$. Note that $\tau_R \to T$ a.s. when $R\to \infty$ due to the well-posedness of \eqref{eq:Galerkin approx}.\par 
    We begin by making the dependence on $Q$ in \cite[inequality (A.3)]{Brehier_Cox_Millet_2024} more explicit. Consider the inequality provided just above \cite[inequality (A.3)]{Brehier_Cox_Millet_2024}, that is, 
    \begin{align*}
        &\EE{\big\|X_M^Q(t \wedge \tau_R)\big\|^p} + p\EE{\int_0^{t \wedge \tau_R}\big\|\nabla X_M^Q(s)\big\|_{L^2}^2 \big\|X_M^Q(s)\big\|_{L^2}^{p-2} ds} \\
        &\leq \EE{\big\|P_M X_0}^p\big\|_{L^2} + \tfrac{1}{4}p^2\tr (Q) \EE{\int_0^t \big( \big\| X_M^Q(s \wedge \tau_R)\big\|_{L^2}^p + 1 \big)ds}.
    \end{align*}
    Since $\|P_M\|_{\mathcal{L}(L^2)}=1$, one has $\Ltwo{P_M X_0} \leq \Ltwo{X_0}$, and hence
    \begin{align*}
        \EE{\bigLtwo{X_M^Q(t \wedge \tau_R)}^p} \leq \EE{\Ltwo{X_0}^p} + \tfrac{1}{4}p^2 T\tr(Q) + \tfrac{1}{4}p^2\tr (Q) \int_0^t \EE{\bigLtwo{X_M^Q(s \wedge \tau_R)}^p} ds.
    \end{align*}
    An application of Grönwall's inequality (see, e.g.,~\cite[Lemma C.1.1]{Dalang_Sanz-Solé_2024}) shows
    \begin{align*}
        \EE{\bigLtwo{X_M^Q(t \wedge \tau_R)}^p} \leq \left(\EE{\Ltwo{X_0}^p} + \tfrac{1}{4}p^2 T\tr(Q)\right)\left(1+e^{\frac{1}{4}p^2 t \tr (Q)}\right).
    \end{align*}
    Taking supremum over $t \in [0,T]$, we deduce
    \begin{equation}
    \label{X_M moment sup bound}
    \begin{aligned}
        \sup_{t\in [0,T]}\EE{\bigLtwo{X_M^Q(t \wedge \tau_R)}^p} &\leq \left(\EE{\Ltwo{X_0}^p} + \tfrac{1}{4}p^2 T\tr(Q)\right)\left(1+e^{\frac{1}{4}p^2 T\tr (Q)}\right)  \\
        &\leq \left(1 + \tfrac{1}{4}p^2 T\tr(Q)\right)\left(1+e^{\frac{1}{4}p^2 T\tr (Q)}\right)(\EE{\Ltwo{X_0}^p} + 1) \\
        &\leq \left(1+e^{\frac{1}{4}p^2 T\tr (Q)}\right)^2(\EE{\Ltwo{X_0}^p} + 1).
    \end{aligned}
    \end{equation}
    Moreover, the following inequality is shown on \cite[p. 26]{Brehier_Cox_Millet_2024}:
    \begin{align}
        &\EE{\sup_{t\in [0,T]}\left(\bigLtwo{X_M^Q(t\wedge \tau_R)}^p + p\int_0^{t\wedge \tau_R} \bigLtwo{\nabla X_M^Q(s)}^2 \bigLtwo{X_M^Q(s)}^{p-2} ds \right)} \label{mom bounds 1}\\
        &\leq \EE{\Ltwo{X_0}^p} + \tfrac{1}{2}\mathbb{E}\Big[\sup_{s\in [0,T]}\bigLtwo{X_M^Q(s\wedge \tau_R)}^p\Big]
        + \tfrac{19}{4}p^2\tr(Q)\EE{\int_0^T \left(\EE{\bigLtwo{X_M^Q(s\wedge \tau_R)}^p}+1\right) ds}. \nonumber
    \end{align}
    Note that $\sup_{s\in [0,T]}\bigLtwo{X_M^Q(s\wedge \tau_R)}^p \leq R^p + \bigLtwo{X_0}^p < \infty$. Hence, it holds that
    \begin{align*}
        \mathbb{E}\Big[\sup_{s\in [0,T]}\bigLtwo{X_M^Q(s\wedge \tau_R)}^p\Big] < \infty,
    \end{align*} 
    so that it can be subtracted from both sides of \eqref{mom bounds 1}. Furthermore, using \eqref{X_M moment sup bound} for the term inside the integral in \eqref{mom bounds 1}, we obtain
    \begin{align*}
        &\EE{\sup_{t\in [0,T]}\left(\tfrac{1}{2}\bigLtwo{X_M^Q(t\wedge \tau_R)}^p + p\int_0^{t\wedge \tau_R} \bigLtwo{\nabla X_M^Q(s)}^2 \bigLtwo{X_M^Q(s)}^{p-2} ds \right)}  \\
        &\leq \EE{\Ltwo{X_0}^p} 
        + \tfrac{19}{4}p^2 T\tr(Q)\left(\left(1+e^{\frac{1}{4}p^2 T\tr (Q)}\right)^2(\EE{\Ltwo{X_0}^p} + 1)+1\right)  \\
        &\leq \left(1 +  \tfrac{19}{4}p^2 T\tr(Q)
        + \tfrac{19}{4}p^2 T\tr(Q)\left(1+e^{\frac{1}{4}p^2 T\tr (Q)}\right)^2\right)(\EE{\Ltwo{X_0}^p} + 1).
    \end{align*}
    Multiplying by a factor of $2$ and taking supremum over $M\in \N, R\in (0,\infty)$, we deduce
    \begin{align}
        &\sup_{M\in \N}\sup_{R \in (0,\infty)}\EE{\sup_{t\in [0,T]}\left(\bigLtwo{X_M^Q(t\wedge \tau_R)}^p + 2p\int_0^{t\wedge \tau_R} \bigLtwo{\nabla X_M^Q(s)}^2 \bigLtwo{X_M^Q(s)}^{p-2} ds \right)} \nonumber \\
        &\leq 2\left(1 +  \tfrac{19}{4}p^2 T\tr(Q)
        + \tfrac{19}{4}p^2 T\tr(Q)\left(1+e^{\frac{1}{4}p^2 T\tr (Q)}\right)^2\right)(\EE{\Ltwo{X_0}^p} + 1). \label{sup moment bound 2}
    \end{align}
    Letting $R\to \infty$, recalling that $\tau_R \to T$ a.s.\ and applying monotone convergence theorem shows the desired inequality, with 
    \begin{align*}
        F_{p,T}(x) = 2\left(1 +  \tfrac{19}{4}p^2 T x
        + \tfrac{19 }{4}p^2T x\left(1+e^{\frac{1}{4}p^2 Tx}\right)^2\right).
    \end{align*}
\end{proof}

\begin{proof}[Proof of Lemma~\ref{lem:L_inf bound of X_M}]
    We take over from the last inequality given in the proof of \cite[Lemma 3.2]{Brehier_Cox_Millet_2024}, which reads
\begin{align*}
    &\sup_{t \in [0,T], z \in [0,1]}|X_M^Q(t)(z)| \\
    &\leq C_\alpha \Ltwo{(-A)^{\alpha}X_0} + 2\sup_{t \in [0,T]}\|I_M(t)\|_{L^\infty}+C_{\Tilde{p}}\left(\int_0^T \bigLtwo{X_M^Q(s)}^{4\Tilde{p}-2}\bigLtwo{\nabla X_M^Q(s)}^2\, ds\right)^{\frac{3}{4\Tilde{p}}},
\end{align*}
where $I_M(t):= \int_0^te^{(t-s)A}P_MdW^Q(s) $ and $ \Tilde{p} = \frac{3p}{4}$. Raising this to $p = \frac{4\Tilde{p}}{3}$ and taking expectations, one obtains
\begin{align}
    &\EE{\sup_{t \in [0,T], z \in [0,1]}|X_M^Q(t)(z)|^p}\label{sup X_M bound 1}\\ &\leq 3^{p}\bigg(C^p_\alpha \EE{\Ltwo{(-A)^{\alpha}X_0}^p} + 2^p\mathbb{E}\Big[\sup_{t \in [0,T]}\|I_M(t)\|^p_{L^\infty}\Big] +C^p_{\Tilde{p}}\EE{\int_0^T \bigLtwo{X_M^Q(s)}^{4\Tilde{p}-2}\bigLtwo{\nabla X_M^Q(s)}^2\, ds} \bigg). \nonumber
\end{align}
By Lemma~\ref{lem:Y Linf bound}, one has
\begin{align}
    \EE{\sup_{t \in [0,T]}\|I_M(t)\|^p_{L^\infty}} \leq C_{p,T}\tr(Q)^{p/2}, \label{I_M sup bound}
\end{align} 
while by Lemma~\ref{lem: X_M sup +exp moment bounds} (i), since $4\Tilde{p} = 3p \geq 8$, there exists an increasing $F_{p,T}:(0,\infty) \to (0,\infty)$ such that
\begin{align}
    \EE{\int_0^T \bigLtwo{X_M^Q(s)}^{4\Tilde{p}-2}\bigLtwo{\nabla X_M^Q(s)}^2\, ds} \leq F_{p,T}(\tr (Q))(\EE{\Ltwo{X_0}^{3p}}+1). \label{X_M p tilde}
\end{align}
Plugging \eqref{I_M sup bound} and \eqref{X_M p tilde} in the inequality \eqref{sup X_M bound 1}, and taking supremum over $M \in\N$,
\begin{align*}
    &\sup_{M \in \N}\EE{\sup_{t \in [0,T], z \in [0,1]}|X_M^Q(t)(z)|^p}\\ &\leq 3^{p}\Big(C^p_\alpha \EE{\Ltwo{(-A)^{\alpha}X_0}^p} + 2^p C_{p,T}\tr(Q)^{p/2}+C^p_{\Tilde{p}}F_{p,T}(\tr (Q))(\EE{\Ltwo{X_0}^{3p}}+1)\Big) \\
    &\leq C_{p,\alpha,T}\max\{1,\tr (Q)^{p/2},F_{p,T}(\tr (Q))\} \left(1+\EE{\Ltwo{(-A)^{\alpha}X_0}^p} + \EE{\Ltwo{X_0}^{3p}}\right),
\end{align*}
for some $C_{p,\alpha,T} \in (0,\infty)$. The function $(0,\infty) \ni x \mapsto \max\{1,x^{p/2},F_{p,T}(x)\}$ is non-decreasing and only depends on $p$ and $T$, finishing the proof.
\end{proof} 

\begin{proof}[Proof of Lemma~\ref{lem:(-A)X_M bound}]
    Recall the mild formulation of $X_M^Q$. Using the triangle inequality in $C^\gamma([0,T],L^2)$ and inequalities (A.20), (A.21) and Lemma A.1 in \cite{Brehier_Cox_Millet_2024}, one has
\begin{align*}
     &\sup_{M \in \N} \EE{\|(-A)^{\lambda}X_M^Q\|^p_{C^\gamma([0,T],L^2)}} \\
     &\leq 3^p\bigg(\EE{\Ltwo{(-A)^{\alpha}X_0}^{p}} + C_{\gamma,\lambda,T}\sup_{M \in \N}\EE{\big\|X_M^Q \big\|^{2p}_{L^\infty([0,T],L^4)}}+ C_{\lambda,\gamma,p,T}\tr (Q)^{p/2}\bigg).
\end{align*}
As $\|\cdot\|_{L^4} \leq \|\cdot\|_{L^\infty}$, by Lemma~\ref{lem:L_inf bound of X_M}
\begin{align*}
    &\sup_{M \in \N} \EE{\|(-A)^{\lambda}X_M^Q\|^p_{C^\gamma([0,T],L^2)}} \\
    &\leq 3^p\bigg(\EE{\Ltwo{(-A)^{\alpha}X_0}^{p}} + C_{\gamma,\lambda,T}C_{p,\alpha,T}F_{p,T}(\tr(Q))\left(1+\EE{\Ltwo{(-A)^{\alpha}X_0}^{2p}} + \EE{\Ltwo{X_0}^{6p}}\right)\\
    &\hspace{1em}+ C_{\lambda,\gamma,p,T}\tr (Q)^{p/2}\bigg)\\
     &\leq C_{p,\alpha,\gamma,\lambda,T}\max\{1,F_{p,T}(\tr(Q)),\tr (Q)^{p/2}\}\left(1+\EE{\Ltwo{(-A)^{\alpha}X_0}^{2p}} + \EE{\Ltwo{X_0}^{6p}}\right),
\end{align*}
for some $C_{p,\alpha,\gamma,\lambda,T} \in (0,\infty)$. The function $(0,\infty) \ni x \mapsto \max\{1,F_{p,T}(x),x^{p/2}\}$ is non-decreasing and only depends on $p$ and $T$, which finishes the proof.
\end{proof}

\begin{proof}[Proof of Lemma~\ref{lem:sup of psi}]
The proof presented here is nearly identical to the one of \cite[Lemma 5.4]{Brehier_Cox_Millet_2024}, with differing parts highlighted in \textcolor{red}{red}. Note that the changes are justified by the more explicit bounds given in Lemma~\ref{lem: X_M sup +exp moment bounds} and Lemma~\ref{lem:(-A)X_M bound}.
    Recall that $\| (-A)^{\alpha} x \|_{L^2} \leq  \| (-A)^{\beta} x\|_{L^2}$ for all $x\in D((-A)^{\beta})$ and all $\alpha<\beta$
(see~\eqref{eq:A_fracpownorminc}). Thus, without loss of generality, we can assume $\delta_0< \frac{1}{4}$. 
Applying H\"older's inequality, one has for all 
$M\in\N$
\begin{align*}
\EE{\sup_{t\in [0,T]} \Psi_{\delta,\epsilon,q}(X_M^Q(t))}\leq &\left(
\EE{\sup_{t\in [0,T]} \exp\big(\tfrac{p\epsilon}{p-2q}\|X_M^Q(t)\|_{L^2}^2 \big)
}\right)^{1-\frac{2q}{p}}\\
&\times \left( \EE{
\sup_{t\in [0,T]} \big(1+\|(-A)^{\frac{1}{4}+\delta } X_M^Q(t)\|_{L^2}^{q}\big)^{\frac{p}{2q}}
}\right)^{\frac{2q}{p}}.
\end{align*}
Applying the exponential moment bounds~\textcolor{red}{\eqref{eq:X_M exp bound random initial}} from Lemma~\textcolor{red}{\ref{lem: X_M sup +exp moment bounds}} (with $\beta=\frac{p\epsilon}{p-2q}$) and the inequality~\textcolor{red}{\eqref{eq:(-A)X_M inequality}} from Lemma~\textcolor{red}{\ref{lem:(-A)X_M bound}} (with $\lambda=\frac14+\delta$, $\gamma \in (0,\delta_0-\delta)$, $\alpha=\frac14+\delta_0$) one obtains that there exists \textcolor{red}{a non-decreasing $F_{p,T}:(0,\infty)\to (0,\infty)$ (dependent on $p$ and $T$) and $C_{\delta,\delta_0,p,T} >0$} such that
\begin{align*}
    &\EE{\sup_{t\in [0,T]} \exp\big(\tfrac{p\epsilon}{p-2q}\|X_M^Q(t)\|_{L^2}^2 \big)}\le \textcolor{red}{2e^{\frac{\gamma_0 T \tr(Q)}{1+2\gamma_0 \|Q \|_{\mathcal{L}(L^2)}}}}\left( \EE{\exp\big(\gamma_0 \|X_0\|_{L^2}^2 \big)}\right)^{\frac{p\epsilon}{(p-2q)\gamma_0}}\\
    &\EE{ \sup_{t\in [0,T]} \big(1+\|(-A)^{\frac{1}{4}+\delta } X_M^Q(t)\|_{L^2}^{q}\big)^{\frac{p}{2q}}}\le \textcolor{red}{C_{\delta,\delta_0,p,T}F_{p,T}(\tr (Q))}\Big( 1+\E \left[\| (-A)^{\frac14+\delta_0} X_0 \|_{L^2}^{p}\right]+ 
    \EE{\| X_0 \|_{L^2}^{3p}}\Big).
\end{align*} 
Note that $\epsilon<\gamma_0$, therefore one has
\[
\left( \EE{\exp\big(\gamma_0 \|X_0\|_{L^2}^2 \big)}\right)^{\frac{\epsilon}{\gamma_0}}\le 1+\EE{\exp\big(\gamma_0 \|X_0\|_{L^2}^2 \big)}.
\]
Moreover, there exists $C_{\gamma_0,p}\in(0,\infty)$ such that
\begin{align*}
1+\E \left[\| (-A)^{\frac{1}{4}+\delta_0} X_0 \|_{L^2}^{p}\right]+  \E \left[\| X_0 \|_{L^2}^{3p}\right]&\le \big(1+\E \left[\| (-A)^{\frac{1}{4}+\delta_0} X_0 \|_{L^2}^{p}\right]\big) \big(1+\EE{\| X_0 \|_{L^2}^{3p}}\big)\\
&\le C_{\gamma_0,p}\big(1+\EE{\| (-A)^{\frac{1}{4}+\delta_0} X_0 \|_{L^2}^{p}}\big)\big(1+\EE{\exp\big(\gamma_0 \|X_0\|_{L^2}^2 \big)}\big).
\end{align*}
Combining the upper bounds then yields the inequality~\textcolor{red}{\eqref{eq:sup of Psi bound}} and concludes the proof of Lemma~\textcolor{red}{\ref{lem:sup of psi}}.
\end{proof}

\end{document}